\documentclass[11pt]{amsart}
\usepackage{amssymb}
\usepackage{mathrsfs}
\usepackage{cite}
\usepackage{algorithm}
\usepackage{algorithmic}
\usepackage{amscd}

{\theoremstyle{definition} 
\newtheorem{dfn}{Definition}
\newtheorem{ex}{Example}}
\newtheorem{pro}{Proposition}
\newtheorem{lmm}{Lemma}
\newtheorem{thm}{Theorem}
\newtheorem{cor}{Corollary}

\begin{document} 

\title{Finite sets of $d$-planes in affine space}
\author{Mathias Lederer}
\email{mlederer@math.uni-bielefeld.de}
\address{Fakult\"at f\"{u}r Mathematik, Universit\"{a}t Bielefeld, Bielefeld, Germany}
\date{September, 2008}

\maketitle

\begin{abstract} 
  Let $A$ be a subvariety of affine space $\mathbb{A}^n$
  whose irreducible components are $d$-dimensional linear or affine subspaces of $\mathbb{A}^n$. 
  Denote by $D(A)\subset\mathbb{N}^n$ the set of exponents of standard monomials of $A$.
  We show that the combinatorial object $D(A)$ reflects the geometry of $A$ in a very direct way.
  More precisely, we define a $d$-plane in $\mathbb{N}^n$ as being a set $\gamma+\oplus_{j\in J}\mathbb{N}e_{j}$, 
  where $\#J=d$ and $\gamma_{j}=0$ for all $j\in J$. 
  We call the $d$-plane thus defined to be parallel to $\oplus_{j\in J}\mathbb{N}e_{j}$.
  We show that the number of $d$-planes in $D(A)$ equals the number of components of $A$.
  This generalises a classical result, the finiteness algorithm, which holds in the case $d=0$.
  In addition to that, we determine the number of all $d$-planes in $D(A)$ parallel to $\oplus_{j\in J}\mathbb{N}e_{j}$,
  for all $J$. 
  Furthermore, we describe $D(A)$ in terms of the standard sets of the intersections $A\cap\{X_{1}=\lambda\}$,
  where $\lambda$ runs through $\mathbb{A}^1$.
\end{abstract}

\section{Introduction}

Let $k$ be a field and $k[X]=k[X_{1},\ldots,X_{n}]$ be the polynomial ring in $n$ variables. 
We fix a term order $<$ on $k[X]$ such that $X_{1}<\ldots<X_{n}$. 
We consider $n$-dimensional affine space $\mathbb{A}^n={\rm Spec}\,k[X]$ over $k$, 
and (a certain class of) ideals $I\subset k[X]$, 
along with the corresponding varieties $V(I)\subset\mathbb{A}^n$. Central objects of study will be the sets
\begin{equation*}
  C(I)=\{{\rm LE}(f);f\in I\}\subset\mathbb{N}^n
\end{equation*}
consisting of {\it leading exponents} of elements of $I$ (with respect to $<$) and its complement 
\begin{equation*}
  D(I)=\mathbb{N}^n-C(I)\,,
\end{equation*}
which is called the set of exponents of {\it standard monomials} of $I$ (see \cite{sturm}), 
or also {\it Gr\"obner \'escalier} of $I$ (see \cite{big2}).
We often shift between the use of monomials and the use of their exponents. 
Therefore, we call $D(I)$ itself the {\it standard set} of $I$.
Clearly, the set $C(I)$ is stable under the canonical action of the additive monoid $\mathbb{N}^n$ on itself. 
Therefore, the set 
\begin{equation*}
  \mathbb{D}_{n}=\{\delta\subset\mathbb{N}^n;\,\,{\rm if }\,\,\alpha\in\mathbb{N}^n-\delta,
  \,\,{\rm then }\,\,\alpha+\beta\in\mathbb{N}^n-\delta,\,\,{\rm for}\,\,{\rm all }\,\,\beta\in\mathbb{N}^n\}
\end{equation*}
consists of all subsets of $\mathbb{N}^n$ which occur as sets of standard monomials of ideals $I\subset k[X]$.
The set $\mathbb{D}_{n}$ will be used throughout the text. 

All ideals under consideration are radical. Therefore, 
we have a correspondence between ideals and their varieties.
We will denote the affine variety attached to an ideal $I$ by $A=V(I)$, 
and conversely, the ideal defining an affine variety $A$ by $I(A)$.
We use the shorthand notation $C(A)=C(I)$ and $D(A)=D(I)$.
Our goal is to describe a connection between varieties $A$ (geometric objects) 
and standard sets $D(A)$ (combinatorial objects). 
This will certainly not be a bijection, 
since the combinatorial objects are much coarser than the geometric objects. 
However, the combinatorial object $D(A)$ will reflect much of the geometry of $A$.

Let us start with a simple and well-known special case.
\begin{pro}\label{finiteness}
  Let $I\subset k[X]$ be a radical ideal.
  Then $D(I)$ is a finite set if, and only if, for all field extensions $k^\prime\supset k$, 
  the set $V_{k^\prime}(I)$ of $k^\prime$-rational closed points of $V(I)$ is finite. In this case, 
  $\#D(I)=\#V_{\overline{k}}(I)$, where $\overline{k}$ is the algebraic closure of $k$. 
\end{pro}

\begin{proof}
  This is due to the Chinese Remainder Theorem 
  and the fact that the standard monomials form a basis of the $k$-vector space $k[X]/I$. 
  See also \cite{cox}, where this proposition is discussed in the context of the {\it finiteness algorithm}. 
\end{proof}

Hence, the combinatorial object $D(A)$ inherits essential information about the geometry of 
variety $A$ consisting of finitely many rational points---Proposition \ref{finiteness} yields the equality
\begin{equation}\label{finiteequality}
  \#D(I)=\#A\,.
\end{equation}
In fact, if we use the lexicographic order on $k[X]$, the shape of the combinatorial object $D(A)$
contains much more information about the shape of $V(I)$ than only its cardinality, 
as is stated by equation (\ref{finiteequality}). For a discussion of this issue, see \cite{jpaa} and references therein.
In the present paper, we want to generalise equation (\ref{finiteequality}) in the following way:
\begin{itemize}
  \item On the combinatorial side, we replace finite sets in $\mathbb{D}_{n}$ by infinite sets in $\mathbb{D}_{n}$. 
  \item On the geometric side, we consider a class of non-closed $k$-rational points of $\mathbb{A}^n$,
  namely, linear or affine $d$-dimensional affine subspaces of $k^n$.
\end{itemize}
We think of these points as a particularly simple kind of nonclosed points in affine space $\mathbb{A}^n$. 
It is clear that in the case where the components of $A$ are linear subspaces of $k^n$,
the variety $A$ can be considered as a finite collection of $k$-rational points in the Grassmannian ${\rm Grass}(d,n)$. 
In Section \ref{vars}, we will explain that also in the case where the components of $A$ are affine subspaces of $k^n$, 
$A$ has an interpretation as a finite collection of $k$-rational points in a Grassmannian.

In particular, both $A$ and $D(A)$ are infinite sets. Our generalisations of equation (\ref{finiteequality}) 
will therefore involve other invariants than just the cardinalities of $A$ and $D(A)$; we will prove the following results.
\begin{itemize}
  \item In Definition \ref{dplane}, we introduce {\it $d$-planes} in $\mathbb{N}^n$,
  which are subsets of the form $\gamma+\oplus_{j\in J}\mathbb{N}e_{j}$,
  where $\#J=d$ and $\gamma_{j}=0$ for all $j\in J$.
  Theorem \ref{number} states in particular that the number of $d$-planes in $D(A)$ equals the number of components of $A$.
  This result is a clearly a generalisation of equation (\ref{finiteequality}), 
  in stating equality of sizes of a combinatorial and a geometric object.
  However, we will refine this assertion in the following ways. 
  \item In Theorem \ref{number}, we also specify, for each $J$ with $\#J=d$, 
  how many $d$-planes $\gamma+\oplus_{j\in J}\mathbb{N}e_{j}$ are contained in $D(A)$.
  For this, a close analysis of the equations defining the components of $A$ is necessary. 
  The key notion here is that of {\it minimal free variables}, as is introduced in Definition \ref{minimal}.
  Note that this result goes beyond what can be said in the case $d=0$---more precisely,
  in that context, the analogoue of our $J$ is the empty set, 
  hence the analogous statement is empty as well.
  \item In the case where all components of $A$ are parallel to the hyperplane $\{X_{1}=0\}$
  (and the term order has the property stated in Definition \ref{prodorder}), 
  we can explicitly compute $D(A)$ in terms of $D(A_{\lambda})$, $\lambda\in\mathbb{A}^1$,
  where $A_{\lambda}$ is the subvariety $A\cap\{X_{1}=\lambda\}$ of $A$. 
  This will be established in Theorem \ref{stack}.
  The crucial operation here is {\it addition of standard sets} (see Definition \ref{addmap}), 
  which has been introduced already in \cite{jpaa}. 
  The statement of Theorem \ref{stack} is stronger than that of Theorem \ref{number}, 
  in describing the combinatorics of $D(A)$ in a much finer way.
  \item If not all components of $A$ are parallel to the hyperplane $\{X_{1}=0\}$, 
  we use the main Theorem of \cite{wibmer} for showing the existence of a Zariski open $U\subset\mathbb{A}^1$
  such that $D(A_{\lambda})$ is constant for all $\lambda\in U$, of value $\delta\subset\mathbb{N}^{n-1}$, say. 
  In Theorem \ref{inherit} and Corollary \ref{corgeneral}, 
  we show that $\mathbb{N}e_{1}\oplus\delta$ is contained in $D(A)$, 
  and that this is the largest subset of $D(A)$ which is a union of $1$-planes $\gamma+\mathbb{N}e_{1}$.
  This result is stronger than Theorem \ref{number}
  since it yields information not only on the highest dimensional subsets of $D(A)$.
\end{itemize}

\section{Minimal free variables}\label{vars}

We now describe the geometric objects of our study. 
Let $A$ be a closed subvariety of $\mathbb{A}^n$ with $m$ irreducible components, 
such that each component $A^\prime$ of $A$ is a $d$-dimensional affine subspace of $k^n$. 
(By an {\it affine} subspace, as opposed to a {\it linear} subspace of $k^n$, 
we understand a $d$-dimensional plane which does not necessarily pass through the origin of $k^n$.)

We embed the affine space $\mathbb{A}^n$ we started with into $\mathbb{A}^{n+1}$ by the map 
\begin{equation}\label{iota}
  \iota:\mathbb{A}^n\to\mathbb{A}^{n+1}:(a_{1},\ldots,a_{n})\mapsto(1,a_{1},\ldots,a_{n})\,.
\end{equation}
Each $d$-dimensional affine subspace $A^\prime$ of $\mathbb{A}^n$ 
defines a $(d+1)$-dimensional linear subspace of $\mathbb{A}^{n+1}$,
namely, the linear space spanned by the elements of $\iota(A^\prime)$. 
Denote by $X(d,n)$ the subset of the Grassmannian ${\rm Grass}(d+1,n+1)$ consisting of all linear $(d+1)$-spaces 
in $\mathbb{A}^{n+1}$ whose intersection with the hyperplane $\{X_{0}=0\}$ of $\mathbb{A}^{n+1}$ is empty.
($X_{0}$ is the additional coordinate we use for embedding $\mathbb{A}^n$ into $\mathbb{A}^{n+1}$.)
Clearly, $X(d,n)$ is Zariski-open in ${\rm Grass}(d+1,n+1)$. 
Upon identifying $A^\prime$ and the span of $\iota(A^\prime)$, 
the space of all $d$-dimensional affine subspaces of $\mathbb{A}^n$ is identified with the space $X(d,n)$.
Thus, the variety $A$ may be considered as an $m$-element set of $k$-rational closed points in $X(d,n)$.

Let us fix a component $A^\prime$ of $A$ and study it in terms of linear equations. 
We think of $A^\prime$ as being an affine $d$-plane in $k^n$, thus, the solution to a linear equation
\begin{equation}\label{linear1}
  BX+c=0\,,
\end{equation}
where $B\in{\rm M}_{n}(k)$ has rank $n-d$, $c\in{\rm M}_{n,1}(k)$, and $X$ is the column with entries $X_{1},\ldots,X_{n}$. 
By the usual operations on the lines of $B$ and $c$ and a permutation of columns of $B$, (\ref{linear1}) is equivalent to
\begin{equation}\label{linear2}
  \widetilde{B}\widetilde{X}+\widetilde{c}=0\,,
\end{equation}
where
\begin{equation*}
  \begin{split}
  \widetilde{B}=\left(\begin{array}{cc} E_{n-d} & * \\ 0 & 0 \end{array}\right)\,,\,\,
  \widetilde{X}=\left(\begin{array}{c} X_{\sigma(1)} \\ \vdots \\ X_{\sigma(n)} \end{array}\right)\,,\,\,
  \widetilde{c}=\left(\begin{array}{c} * \\ 0 \end{array}\right)\,.
  \end{split}
\end{equation*}
($E_{n-d}$ denotes the $(n-d)\times(n-d)$-unit matrix.)
The variables $\widetilde{X}_{n-d+1},\ldots,\widetilde{X}_{n}$ are sometimes called {\it free variables} of $A^\prime$, 
since they can take arbitrary values, 
whereas the values of $\widetilde{X}_{1},\ldots,\widetilde{X}_{d}$ are uniquely determined by the choice
of values of the free variables. 
However, the set of free variables of an affine plane is not a well-defined quantity. 
If, e.g., a hyperplane is defined by the equation $B_{1}X_{1}+\ldots+B_{n}X_{n}+c=0$, and $B_{1}\ldots B_{n}\neq0$, 
then each $(n-1)$-element subset of $\{X_{1},\ldots,X_{n}\}$ is a set of free variables. 

\begin{dfn}\label{minimal}
  Let $A^\prime$ be a $d$-dimensional affine subspace of $\mathbb{A}^n$ and $J\subset\{1,\ldots,n\}$ such that $\#J=d$. 
  Then the elements of the set $\{X_{j};j\in J\}$ are called {\rm minimal free variables} if 
  $\{X_{j};j\in J\}$ is a set of free variables of $A^\prime$ and for all $j\in J$, 
  there exists no $i\in\{1,\ldots,n\}-J$, $i<j$, 
  such that for $J^\prime=(J-\{j\})\cup\{i\}$, 
  the set $\{X_{j};j\in J^\prime\}$ is a set of free variables of $A^\prime$.
\end{dfn}

By definition, a set of minimal free variables is unique. The name, minimal, 
reflects the fact that we have $X_{1}<\ldots<X_{n}$.
Before explaining the significance of minimal free variables to our situation, 
let us give this notion another characterisation, in terms of a recursion,
and let us find parameters which uniquely determine $A^\prime$.

Let $\xi$ be any solution of (\ref{linear1}). 
The set of solutions of (\ref{linear1}) is in bijection with the set of solutions of
\begin{equation}\label{by}
  BY=0\,,
\end{equation}
via $X=Y+\xi$. Define $Y_{1}=1$, and consider the equation
\begin{equation}\label{linear3}
  B\left(\begin{array}{cc} 1 \\ Y_{2} \\ \vdots \\ Y_{n} \end{array}\right)=0\,,
\end{equation}
which is in fact an inhomogeneous system in the variables $Y_{2},\ldots,Y_{n}$. 
\begin{itemize}
  \item If (\ref{linear3}) has a solution, then $X_{1}$ is one of the minimal free variables. 
  Proceed by induction over $n$: In the next step, define $Y_{1}=0$, and consider the affine $(d-1)$-plane 
  $A^\prime\cap\{X_{1}=\xi_{1}\}$ in $\mathbb{A}^{n-1}$ defined by (\ref{by}).
  \item If (\ref{linear3}) has no solution, then $X_{1}$ is not one of the minimal free variables. 
  In this case, $A^\prime\subset\{X_{1}=\xi_{1}\}=\mathbb{A}^{n-1}$. 
  The plane $A^\prime$ is characterized by (\ref{by}), where $Y_{1}=0$. 
  Proceed by induction over $n$.
\end{itemize}

\begin{pro}\label{unique}
  Let $\{X_{j};j\in J\}$ be minimal free variables of $A^\prime$. 
  Then there exists a unique system of equations defining $A^\prime$,
  \begin{equation}\label{eqs}
    X_{i}+\sum_{j\in J,j<i}b_{i,j}X_{j}+c_{i}=0\,,
  \end{equation}
  for all $i\in\{1,\ldots,n\}-J$.
\end{pro}

\begin{proof}
  Consider system (\ref{linear2}), which defines $A^\prime$. We choose a permutation $\sigma$ in such a way that 
  $(X_{\sigma(1)},\ldots,X_{\sigma(n-d)})=(X_{i})_{i\notin J}$. The equations in (\ref{linear2}) are 
  \begin{equation*}
    X_{i}+\sum_{j\in J}b_{i,j}X_{j}+c_{i}=0\,,\,\,{\rm for}\,\,{\rm all }\,\,i\notin J\,.
  \end{equation*}
  We claim that this is in fact (\ref{eqs}), i.e., $b_{i,j}=0$ for all pairs $(i,j)$ such that $i\notin J$, $j\in J$ and $i<j$.
  Indeed, if there exist such $i,j$ with $b_{i,j}\neq0$, we can interchange those columns of 
  $\widetilde{B}$ which correspond to the variables $X_{i}$ and $X_{j}$. We get
  \begin{equation}\label{linear4}
    \left(\begin{array}{cc} B^\prime & * \\ 0 & 0 \end{array}\right)
    \widetilde{X}^\prime+\widetilde{c}=0\,,
  \end{equation}
  where 
  \begin{equation*}
    B^\prime=\left(\begin{array}{ccccccc} 1 & & & * & & & \\ & \ddots & & \vdots & & & \\ & & 1 & * & & & \\ & & & b_{i,j} & & & \\
    & & & * & 1 & & \\ & & & \vdots & & \ddots & \\ & & & * & & & 1 \end{array}\right)\,,
  \end{equation*}
  and $\widetilde{X}^\prime$ arises from $\widetilde{X}$ by interchanging $X_{i}$ and $X_{j}$. 
  Upon transforming the rows of (\ref{linear4}), we arrive at a system
  \begin{equation*}
    \left(\begin{array}{cc} E_{d} & * \\ 0 & 0 \end{array}\right)\widetilde{X}^\prime+\widetilde{c}^\prime=0\,.
  \end{equation*}
  This means that for $J^\prime=(J-\{j\})\cup\{i\}$, also $\{X_{j};j\in J^\prime\}$ are free variables, 
  a contradiction to minimality. Uniqueness is clear. 
\end{proof}

Given $A^\prime$, with minimal free variables $\{X_{j};j\in J\}$, we can think of $A^\prime$ as having ``coordinates''
$(\xi_{1},\ldots,\xi_{n})$, where $\xi_{j}=X_{j}$ if $j\in J$, and $\xi_{i}=-\sum_{j\in J,j<i}b_{i,j}X_{j}-c_{i}$ if $i\notin J$. 

For each $J\subset\{1,\ldots,n\}$ such that $\#J=d$, let $X(J,n)$ denote the subvariety of $X(d,n)$ 
consisting of all affine $d$-planes in $\mathbb{A}^n$ with minimal free variables $\{X_{j};j\in J\}$. 
By Proposition \ref{unique}, each element of $X(J,n)$ can be uniquely written as the solution of a system
\begin{equation*}
  BY+c=0\,.
\end{equation*}
Here $(Y_{i})_{i=1,\ldots,n-d}=(X_{j})_{j\notin J}$ and $(Y_{i})_{i=n-d+1,\ldots,n}=(X_{j})_{j\in J}$. 
Further, $B=\left(\begin{array}{cc} E_{n-d} & b \end{array}\right)$,
where the rows of $b$ are indexed by $\{1,\ldots,n\}-J$, and the columns of $b$ are indexed by $J$. 
Denote by $r(J)$ the sum of all $\#\{j\in J;j<i\}$, where $i$ runs through $\{1,\ldots,n\}-J$. 
Since for all $i\notin J$, we have $b_{i,j}=0$ whenever $j\geq i$, and $c$ is arbitrary, 
the set $X(J,n)$ is isomorphic to $\mathbb{A}^{r(J)+n-d}$. 

In fact, $X(J,n)$ is a locally closed stratum in $X(d,n)$. 
For seeing this, we adopt some notation of the Introduction of \cite{lafforgue}.
For all $J^\prime\subset\{1,\ldots,n\}$ with $\#J^\prime=d$, 
let $Y(J^\prime,n)$ be the subspace of $X(d,n)$ consisting of all $A^\prime$ with free variables $\{X_{j};j\in J^\prime\}$. 
Define the {\it matroid} $(d_{I})_{I\subset\{1,\ldots,n\}}$ by $d_{I}=\#(I\cap J^\prime)$
for all $I\subset\{1,\ldots,n\}$. Then we have
\begin{equation*}
  Y(J^\prime,n)=\{A^\prime\subset k^n;\,\dim(A\cap(\oplus_{i\in I}ke_{i}))=d_{I}\,\,{\rm for}\,\,{\rm all }\,\,I\subset\{1,\ldots,n\}\}\,.
\end{equation*}
As was remarked in \cite{lafforgue}, the space $Y(J^\prime,n)$ associated to $d$ 
is a locally closed stratum of ${\rm Grass}(d+1,n+1)$ (hence also a locally closed stratum of $X(d,n)$), 
called ``cellule de Schubert mince''. It follows that also 
\begin{equation*}
  X(J,n)=Y(J,n)-\cup_{J^\prime}Y(J^\prime,n)
\end{equation*}
is a locally closed stratum of $X(d,n)$, 
where the union goes over all $J^\prime=(J-\{j\})\cup\{j^\prime\}$, for all $j\in J$, 
$j^\prime\in\{1,\ldots,n\}-J$, such that $j^\prime<j$.

Note that the isomorphism $X(J,n)\to\mathbb{A}^{r(J)+n-d}$ in terms of the system of equations 
of Proposition \ref{unique} is nothing but the choice of Pl\"ucker coordinates on the open part 
$X(d,n)$ of the Grassmannian ${\rm Grass}(d+1,n+1)$. 
In the case where $A^\prime$ is a linear space, 
the Pl\"ucker coordinates are in fact the unspecified entries of matrix $\widetilde{B}$ in (\ref{linear2}).
In this case, the stratum $Y(J,n)$ corresponds to the set $U_{J}$ in the notation of  \cite{griff}, Chapter I, Section 5.
See also \cite{hodge}, Chapter XIV, Section 1, 
though in this book, the term ``Pl\"ucker coordinates'' is never used.

\section{The highest dimensional subset of $D(A)$}\label{highest}

Let $A$ be an affine variety as in the previous section. 
In this section, we give a first description of the set of standard monomials $D(A)$. 
For doing so, we have to find an invariant attached to an infinite $\delta\in\mathbb{D}_{n}$,
which will play an analogous role as the number of elements of a finite $\delta\in\mathbb{D}_{n}$. 
\begin{dfn}\label{dplane}
  Let $\delta\in\mathbb{D}_{n}$. 
  A {\rm$d$-plane} in $\delta$ is a subset of $\delta$ of the form $\gamma+\oplus_{i\in J}\mathbb{N}e_{i}$, 
  where $J\subset\{1,\ldots,n\}$ contains $d$ elements, 
  $e_{i}$ is the $i$-th standard basis vector of $\mathbb{N}^n$, and $\gamma_{j}=0$ for all $j\in J$. 
  We say that this $d$-plane is {\rm parallel to} $\oplus_{i\in J}\mathbb{N}e_{i}$.
  Further, given $\delta\in\mathbb{D}_{n}$, there exists a maximal $d$ such that $\delta$ contains a $d$-plane;
  define $E(\delta)$ to be the union of all $d$-planes contained in $\delta$. 
\end{dfn}

Thus, $E(\delta)$ has the same $d$-dimensional parts as $\delta$ and forgets all parts of lower dimension. 
In the case where $\delta=D(I)$ or $\delta=D(A)$, we write $E(\delta)=E(I)$ and $E(\delta)=E(A)$, resp.

\begin{thm}\label{number}
  Let $A\subset\mathbb{A}^n$ be an affine variety whose irreducible components are affine $d$-planes in $\mathbb{A}^n$. 
  For all $J\subset\{1,\ldots,n\}$ such that $\#J=d$, 
  let $m_{J}$ be the number of irreducible components of $A$ having minimal free variables $\{X_{j};j\in J\}$. 
  Then for all such $J$, the number of $d$-planes in $D(A)$ parallel to $\oplus_{i\in J}\mathbb{N}e_{i}$ equals $m_{J}$.
\end{thm}

Let us reduce the assertion of the theorem to a few special cases, 
the investigation of which will enable us to prove Theorem \ref{number}. 

\begin{lmm}\label{decompose}
  For all $J\subset\{1,\ldots,n\}$ such that $\#J=d$, 
  let $A_{J}$ be the subvariety of $A$ consisting of all irreducible components of $A$ 
  whose minimal free variables are $\{X_{j};j\in J\}$. 
  Assume that the assertion of Theorem \ref{number} holds for all $A_{J}$. Then it also holds for $A$. 
\end{lmm}

\begin{proof}
  We show that $E(A)=\cup_{J}E(A_{J})$, where the union goes over all $J$ with $\#J=d$. 
  One inclusion is immediate: Since $I(A)\subset I(A_{J})$ for all $J$, it follows that $C(A)\subset C(A_{J})$ for all $J$, 
  hence $C(A)\subset \cap_{J}C(A_{J})$. Taking complements, we get $D(A)\supset\cup_{J}D(A_{J})$, 
  hence, in particular, also $E(A)\supset\cup_{J}E(A_{J})$. 
  
  As for the other inclusion, we have to show that $\mathbb{N}^n-E(A)\supset\mathbb{N}^n-\cup_{J}E(A_{J})$. 
  Take an arbitrary $\alpha$ on the right hand side. 
  We have to show that $\alpha$ lies also in the left hand side, 
  which means that for all $J$ containing $d$ elements,
  there exists $\delta_{J}\in\oplus_{j\in J}\mathbb{N}e_{j}$ 
  such that $\alpha+\delta_{J}\in\mathbb{N}^n-D(A)=C(A)$. 
  
  We have $\alpha\in\mathbb{N}^n-\cup_{J}E(A_{J})=\cap_{J}(\mathbb{N}^n-E(A_{J}))$.
  We fix one $J$ and consider the inclusion $\alpha\in\mathbb{N}^n-E(A_{J})$.
  By hypothesis, the assertion of Theorem \ref{number} holds for $A_{J}$, 
  hence $E(A_{J})$ consists solely of $d$-planes parallel to $\oplus_{j\in J}\mathbb{N}e_{j}$. Therefore, 
  there exists a $\beta_{J}\in\oplus_{j\in J}\mathbb{N}e_{j}$ such that $\alpha+\beta_{J}\in\mathbb{N}^n-D(A_{J})=C(A_{J})$. 
  In particular, there exists an $f_{J}\in I(A_{J})$ such that ${\rm LE}(f_{J})=\alpha+\beta_{J}$. 
  
  Next, consider an arbitrary $J^\prime\neq J$ with $\#J^\prime=d$. 
  Since the assertion of Theorem \ref{number} also holds for $A_{J^\prime}$, 
  all $d$-planes in $D(A_{J^\prime})$ are parallel to $\oplus_{j\in J^\prime}\mathbb{N}e_{j}$.
  In particular, the plane $\oplus_{j\in J}\mathbb{N}e_{j}$ is not contained in $D(A_{J^\prime})$. 
  Therefore, there exists a $\gamma_{J^\prime}\in\oplus_{j\in J}\mathbb{N}e_{j}$ which also lies in $C(A_{J^\prime})$. 
  Hence, there exists an $f_{J^\prime}\in I(A_{J^\prime})$ whose leading exponent equals $\gamma_{J^\prime}$. 
  Consider
  \begin{equation*}
    f=f_{J}\prod_{J^\prime\neq J}f_{J^\prime}\in I(A)\,,
  \end{equation*}
  then for the leading exponents of $f$, we have
  \begin{equation*}
    {\rm LE}(f)={\rm LE}(f_{J})+\sum_{J^\prime\neq J}{\rm LE}(f_{J^\prime})
    =\alpha+\beta_{J}+\sum_{J^\prime\neq J}\gamma_{J^\prime}=\alpha+\delta_{J}\in C(A)\,,
  \end{equation*}
  where $\delta_{J}\in\oplus_{j\in J}\mathbb{N}e_{j}$, as desired.
\end{proof}

The lemma provides a first reduction in the proof of Theorem \ref{number}. 
For reducing the statement further, 
we consider the particular case in which all components $A^\prime$ of $A$ are in fact linear spaces.
More precisely, we draw our attention to the following two statements. 

\begin{itemize}
  \item $\mathcal{A}(d,n)$: 
  The assertion of Theorem \ref{number} holds if all irreducible components of $A$ are affine $d$-planes. 
  \item $\mathcal{L}(d,n)$: 
  The assertion of Theorem \ref{number} holds if all irreducible components of $A$ are linear $d$-planes. 
\end{itemize}

\begin{pro}\label{iff}
  For all $d$ and $n$, we have $\mathcal{A}(d,n)$ if, and only if, for all $d$ and $n$, we have $\mathcal{L}(d,n)$.
\end{pro}

\begin{proof}
  Only the ``if'' direction needs a proof. Let $A$ be a variety as in assertion $\mathcal{A}(d,n)$. 
  By Lemma \ref{decompose}, 
  we may assume that the minimal free variables of each component of $A$ are $\{X_{j};j\in J\}$, for a fixed $J$. 
  Denote by $\widehat{I}$ the homogenisation of the ideal $I=I(A)\subset k[X]$ in the polynomial ring $k[X_{0},X]$, 
  and denote by $\widehat{A}\subset\mathbb{A}^{n+1}={\rm Spec}\,k[X_{0},X]$ the corresponding variety. 
  Clearly, each irreducible component of $\widehat{A}$ is the linear $(d+1)$-space spanned by $\iota(A^\prime)$, 
  where $A^\prime$ is an irreducible component of $A$, and $\iota$ is the map (\ref{iota}).
  One easily checks that the minimal free variables of the irreducible components of 
  $\widehat{A}$ are $\{X_{j};j\in \widehat{J}\}$, 
  where $\widehat{J}=J\cup\{0\}$. 
  
  We define a term order $\prec$ on $k[X_{0},X]$ by 
  $X_{0}^{\alpha_{0}}X^\alpha\prec X_{0}^{\beta_{0}}X^\beta$ if
  either $\alpha<\beta$, or $\alpha=\beta$ and $\alpha_{0}<\beta_{0}$. 
  Then clearly $X_{0}\prec\ldots\prec X_{n}$, hence, 
  the term order $\prec$ on $k[X_{0},X]$ has an analogous formal property as the term order 
  $<$ on $k[X]$ we have been working with throughout.
  We may apply assertion $\mathcal{L}(d+1,n+1)$ to the variety $\widehat{A}$, 
  computing $D(\widehat{A})$ w.r.t. $\prec$. 
  Thus, the set $D(\widehat{A})\subset\mathbb{N}^{n+1}$ 
  contains as many $(d+1)$-planes as $A$ has irreducible components, say $m$,
  and each of these $(d+1)$-planes is parallel to $\oplus_{j\in\widehat{J}}\mathbb{N}e_{j}$. 
  
  Let $\alpha^{(\ell)}+\oplus_{j\in\widehat{J}}\mathbb{N}e_{j}$, for $\ell=1,\ldots,m$, 
  be the $(d+1)$-planes in $D(\widehat{A})$.
  We show that the $d$-planes in $D(A)$ are $p(\alpha^{(\ell)})+\oplus_{j\in J}\mathbb{N}e_{j}$, 
  for $\ell=1,\ldots,m$, where $p$ is the projection
  \begin{equation*}
    p:\mathbb{N}^{n+1}\to\mathbb{N}^n:(\alpha_{0},\ldots,\alpha_{n})\mapsto(\alpha_{1},\ldots,\alpha_{n})\,.
  \end{equation*}
  (Therefrom, the assertion of the proposition is immediate.)
  
  On the one hand, each $p(\alpha^{(\ell)})+\oplus_{j\in J}\mathbb{N}e_{j}$ is contained in $D(A)$. 
  Otherwise, there exists a $\beta$ in some $p(\alpha^{(\ell)})+\oplus_{j\in J}\mathbb{N}e_{j}$ and a 
  $g\in I(A)$ with ${\rm LE}(g)=\beta$. By definition of $\prec$, the homogenisation of $g$, call it $f$, 
  has ${\rm LE}(f)\in\alpha^{(\ell)}+\oplus_{j\in\widehat{J}}\mathbb{N}e_{j}\subset C(\widehat{A})$, a contradiction.
  
  On the other hand, $D(A)$ contains no $d$-planes other than $p(\alpha^{(\ell)})+\oplus_{j\in J}\mathbb{N}e_{j}$,
  for $\ell=1,\ldots,m$. Indeed, assume that $\beta+\oplus_{j\in J^\prime}\mathbb{N}e_{j}$ is contained in $D(A)$,
  for some $J^\prime$ with $\#J^\prime=d$ and some $\beta\in\mathbb{N}^n$.
  In particular, for all $\gamma\in\beta+\oplus_{j\in J^\prime}\mathbb{N}e_{j}$, 
  there exists no $g\in I(A)$ with leading exponent $\gamma$. 
  We claim that 
  \begin{equation}\label{subs}
    (0,\beta)+\oplus_{j\in\widehat{J}^\prime}\mathbb{N}e_{j}\subset D(\widehat{A})\,,
  \end{equation}
  where $\widehat{J}^\prime=J^\prime\cup\{0\}$. Otherwise, 
  there exists a $(\gamma_{0},\gamma)\in(0,\beta)+\oplus_{j\in\widehat{J}^\prime}\mathbb{N}e_{j}$ 
  and an $f\in I(\widehat{A})$ with ${\rm LE}(f)=(\gamma_{0},\gamma)$. 
  Since the ideal $I(\widehat{A})$ is homogeneous, all homogenous components of $f$ lie in $I(\widehat{A})$. 
  Upon replacing $f$ by its homogeneous component of highest total degree, we may assume that $f$ is itself homogeneous.
  Then clearly $g=f(1,X)\in I(A)$, and by definition of $\prec$, we have ${\rm LE}(g)=\gamma$, 
  a contradiction. Inclusion (\ref{subs}) is proved, and shows that there exists an $\ell$ such that
  \begin{equation*}
    (0,\beta)+\oplus_{j\in\widehat{J}^\prime}\mathbb{N}e_{j}=\alpha^{(\ell)}+\oplus_{j\in\widehat{J}}\mathbb{N}e_{j}\,,
  \end{equation*}
  hence also 
  \begin{equation*}
    \beta+\oplus_{j\in J^\prime}\mathbb{N}e_{j}=p(\alpha^{(\ell)})+\oplus_{j\in J}\mathbb{N}e_{j}\,.
  \end{equation*}
\end{proof}

\begin{pro}\label{homogeneous}
  For all $d$ and $n$, the statement $\mathcal{L}(d,n)$ is true.
\end{pro}

\begin{proof}
  Let $A$ be a variety as in assertion $\mathcal{L}(d,n)$. 
  As above, we may assume that the minimal free variables of each component of $A$ are $\{X_{j};j\in J\}$, for a fixed $J$.
  Take any $J^\prime\subset\{1,\ldots,n\}$ with $\#J^\prime=d$ and $J^\prime\neq J$. 
  Then there exists an $\ell\in J^\prime-J$, and, by Proposition \ref{unique}, 
  for each irreducible component $A^{(i)}$ of $A$, an equation 
  \begin{equation*}
    X_{\ell}+\sum_{j\in J,j<\ell}b^{(i)}_{\ell,j}X_{j}=0
  \end{equation*}
  defining $A^{(i)}$. Consider the product
  \begin{equation*}
    f=\prod_{i=1}^m(X_{\ell}+\sum_{j\in J,j<\ell}b^{(i)}_{\ell,j}X_{j})\,,
  \end{equation*}
  where $m$ is the number of irreducible components of $A$, 
  then clearly $f\in I(A)$ and ${\rm LE}(f)=me_{\ell}\in\oplus_{j\in J^\prime}\mathbb{N}e_{j}$. 
  Therefore, all $d$-planes in $D(A)$ are parallel to $\oplus_{j\in J}\mathbb{N}e_{j}$. 
  By well-known properties of the Hilbert function, the set $D(A)$ contains precisely $m$ $d$-planes, 
  see \cite{cox2}. 
\end{proof}

Propositions \ref{iff} and \ref{homogeneous} prove Theorem \ref{number}. 

Now that we have derived $\mathcal{A}(d,n)$ from $\mathcal{L}(d,n)$,
and have proved the latter by a very classical token (the Hilbert function), 
the reader might ask why $\mathcal{A}(d,n)$ is remarkable at all. 
However, in Section \ref{families}, we will study the standard monomials of varieties as in $\mathcal{L}(d,n)$
by methods for which the use of varieties as in $\mathcal{A}(d-1,n-1)$ is essential.

Theorem \ref{number} is indeed a higher dimensional analogue of Proposition \ref{finiteness}: 
The variety $A$ is composed by $m=\sum_{J}m_{J}$ affine planes of dimension $d$, 
and accordingly, the set of standard monomials $D(A)$ is composed by $m=\sum_{J}m_{J}$ planes of dimension $d$. 
Additionally, Theorem \ref{number} specifies the directions of the $d$-planes in $D(A)$ 
in terms of the directions of the components of $A$
(by means of the minimal free variables of the components). 
Note that Theorem \ref{number} does not claim that $D(A)$ consists solely of $d$-planes. 
In general, $D(A)$ will also contain lower-dimensional planes not contained in any $d$-plane. 
Here is an example for this.

\begin{ex}\label{ex1}
Take the graded lexicographic order on $\mathbb{Q}[X,Y,Z]$ such that $X<Y<Z$.
Let $A$ be the subvariety of $\mathbb{A}^3$ over $\mathbb{Q}$ with components $A^{(1)}$ and $A^{(2)}$, 
given by the Gr\"obner bases of their ideals,
\begin{equation*}
  \begin{split}
  I^{(1)}&=(Y-X,Z-1)\,,\\
  I^{(2)}&=(X,Z-Y)\,.
  \end{split}
\end{equation*}
Figure \ref{ex1im1} shows a picture of $A$, along with the hyperplanes $\{Z=1\}$ and $\{X=0\}$ in which
the components $A^{(1)}$, resp. $A^{(2)}$, lie.
The minimal free variable of $A^{(1)}$ is $X$, and the minimal free variable of $A^{(2)}$ is $Y$. 
The respective standard sets are $D(A^{(1)})=\mathbb{N}e_{1}$ and $D(A^{(2)})=\mathbb{N}e_{2}$.
The ideal of $A$ has the Gr\"obner basis
\begin{equation*}
  \begin{split}
  I(A)=&(YX-X^2,ZX-YX+X^2-X,\\
  &ZY-Y^2+YX-X,Z^2-ZY+ZX-Z+Y-X)\,.
  \end{split}
\end{equation*}
From the Gr\"obner basis, 
we deduce that the standard $D(A)$ contains the axes $\mathbb{N}e_{1}$ and $\mathbb{N}e_{2}$, 
and also the isolated element $(0,0,1)$, see Figure \ref{ex1im2}. 
In the picture, the solid blocks parallel to $e_{1}$ and $e_{2}$ actually go to infinity.
\end{ex}

\begin{center}
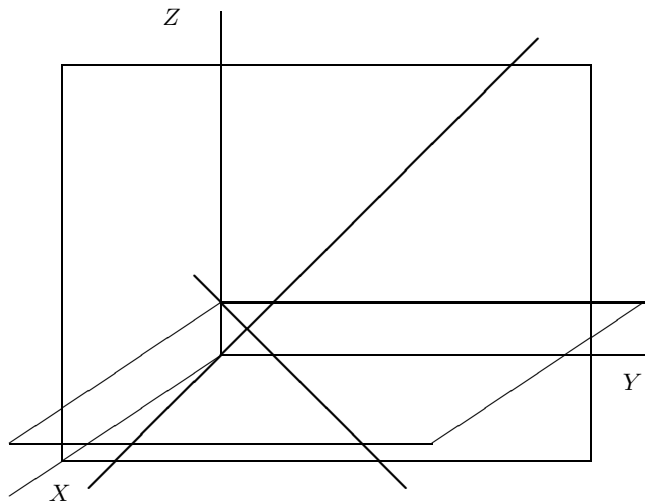
\begin{figure}
  \begin{picture}(400,220)
     \put(130,60){\line(-3,-2){80}}
     \put(65,5){{\footnotesize $X$}}
     \put(130,60){\line(1,0){160}}
     \put(282,47){{\footnotesize $Y$}}
     \put(130,60){\line(0,1){130}}
     \put(108,185){{\footnotesize $Z$}}
     \put(70,20){\line(1,0){200}}
     \put(270,20){\line(0,1){150}}
     \put(270,170){\line(-1,0){200}}
     \put(70,170){\line(0,-1){150}}
     \put(130,80){\line(1,0){160}}
     \put(130,80){\line(-3,-2){80}}
     \put(290,80){\line(-3,-2){80}}
     \put(50,26.5){\line(1,0){160}}
     \thicklines
     \put(130,60){\line(1,1){120}}
     \put(130,60){\line(-1,-1){50}}
     \put(130,80){\line(1,-1){70}}
     \put(130,80){\line(-1,1){10}}
  \end{picture}
\caption{The variety $A$ of Example \ref{ex1}}
\label{ex1im1}
\end{figure}
\end{center}

\begin{center}
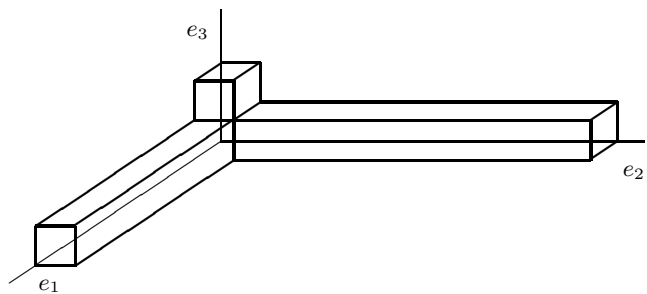
\begin{figure}
  \begin{picture}(400,140)
     \put(130,60){\line(-3,-2){80}}
     \put(61,4){{\footnotesize $e_{1}$}}
     \put(130,60){\line(1,0){160}}
     \put(282,47){{\footnotesize $e_{2}$}}
     \put(130,60){\line(0,1){50}}
     \put(117,100){{\footnotesize $e_{3}$}}
     \thicklines
     \put(145,75){\line(-3,-2){70}}
     \put(120,68){\line(-3,-2){60}}
     \put(135,53){\line(-3,-2){60}}
     \put(75,13){\line(0,1){15}}
     \put(75,13){\line(-1,0){15}}
     \put(60,13){\line(0,1){15}}
     \put(60,28){\line(1,0){15}}
     \put(145,75){\line(1,0){135}}
     \put(280,75){\line(-3,-2){10}}
     \put(280,75){\line(0,-1){15}}
     \put(280,60){\line(-3,-2){10}}
     \put(270,68){\line(0,-1){15}}
     \put(270,68){\line(-1,0){150}}
     \put(270,53){\line(-1,0){135}}
     \put(135,53){\line(0,1){30}}
     \put(145,75){\line(0,1){15}}
     \put(120,68){\line(0,1){15}}
     \put(120,83){\line(1,0){15}}
     \put(120,83){\line(3,2){10}}
     \put(145,90){\line(-1,0){15}}
     \put(145,90){\line(-3,-2){10}}
  \end{picture}
\caption{The standard set of $A$ in Example \ref{ex1}}
\label{ex1im2}
\end{figure}
\end{center}

Thus, in our example, the set $D(A)$ is not the same as $E(A)$, 
but also contains ``lower-dimensional artifacts'', 
by which we understand the $d^\prime$-planes in the difference $D(A)-E(A)$, for all $d^\prime<d$.
Of course, if $A$ consists of only one component, then $D(A)=E(A)$. Thus, 
lower-dimensional artifacts arise from the amalgamation of different irreducible components into $A$.
In the rest of the paper, we find various sources from which lower-dimensional artifacts arise.

Note that in the proof of Lemma \ref{decompose}, we did not show $D(A)=\cup_{J}D(A_{J})$, 
but only the weaker assertion $E(A)=\cup_{J}E(A_{J})$. 
This deficit allows the possibiliy of lower-dimensional artifacts in the case where 
different components of $A$ have different minimal free variables, as in Example \ref{ex1}.
However, also in cases where some irreducible components of $A$ have the same minimal free variables, 
$D(A)$ will contain lower-dimensional artifacts. 
We will discuss such cases in the forthcoming sections.
They are motivated by the following special cases for the dimension of $A$, 
which are particularly easy to understand.
\begin{itemize}
  \item If $d=0$, then trivially, no lower-dimensional artifacts occur.
  \item If $d=n$, we have $A=\mathbb{A}^n$, and trivially, no lower-dimensional artifacts occur.
  \item If $d=n-1$, then each irreducible component of $A$ is an affine hyperplane, 
  hence given by one polynomial of degree $1$, and $A$ is given by the product of these. 
  Therefore, no lower-dimensional artifacts occur. More precisely, if for all $i=1,\ldots,n$, 
  the variety $A$ has $m_{i}$ irreducible components with minimal free variables $\{X_{j};j\in\{1,\ldots,n\}-\{i\}\}$, 
  then
  \begin{equation}\label{cup}
    D(A)=\cup_{i=1}^n(\cup_{\ell=0}^{m_{i}-1}(\ell e_{i}+\oplus_{j\in\{1,\ldots,n\}-\{i\}}\mathbb{N}e_{j}))\,.
  \end{equation}
\end{itemize}
This suggests to use some induction over $n$ and/or $d$. 
More precisely, we will consider the family of intersections
\begin{equation*}
  A_{\lambda}=A\cap\{X_{1}=\lambda\}\subset\mathbb{A}^{n-1}\,,
\end{equation*}
where $\lambda$ runs through all closed points of $\mathbb{A}^1$. 
Here, we identify the hyperplane $\{X_{1}=\lambda\}$ of $\mathbb{A}^{n}$ with 
$\mathbb{A}^{n-1}={\rm Spec}\,k[\overline{X}]$, where $\overline{X}=(X_{2},\ldots,X_{n})$.

Our variety $A$ will have $m$ irreducible components $A^{(1)},\ldots,A^{(m)}$, 
where the $\ell$-th component has miminal free variables $\{X_{j};j\in J^{(\ell)}\}$. 
Two cases will be treated separately. 
\begin{itemize}
  \item $1$ is not contained in any $J^{(\ell)}$. In this case, 
  there is a finite set $Y\subset\mathbb{A}^1$ such that for all $\lambda\in Y$,
  the intersection $A_{\lambda}$ is a variety consisting of $d$-dimensional affine planes in $\mathbb{A}^{n-1}$, 
  and for all $\lambda\in\mathbb{A}^1-Y$, the intersection $A_{\lambda}$ is empty. 
  This case will be studied in Section \ref{interpolation}.
  \item $1$ is contained in all $J^{(\ell)}$. In this case, 
  each intersection $A_{\lambda}$ is a variety consisting of $(d-1)$-dimensional affine planes in $\mathbb{A}^{n-1}$.
  This case will be studied in Section \ref{families}.
\end{itemize}

Finally, in Section \ref{general}, the results of Sections \ref{interpolation} and \ref{families} 
will be applied to the study of the general case, 
i.e. the case where we do not assume any restrictions on the various $J^{(\ell)}$. 
Our arguments will require the term order $<$ to have a property similar to the 
property of term order $\prec$ used above. 
Here, and in the rest of the article, $p$ denotes the projection
\begin{equation*}
  p:\mathbb{N}^n\to\mathbb{N}^{n-1}:(\alpha_{1},\ldots,\alpha_{n})\mapsto(\alpha_{2},\ldots,\alpha_{n})\,.
\end{equation*}
We use the same notation for the projection
\begin{equation*}
  p:\mathbb{A}^n\to\mathbb{A}^{n-1}:(a_{1},\ldots,a_{n})\mapsto(a_{2},\ldots,a_{n})\,.
\end{equation*}

\begin{dfn}\label{prodorder}
  A term order $<$ on $k[X]$ is called a {\rm product order} if
  for all $\alpha=(\alpha_{1},p(\alpha))$ and $\beta=(\beta_{1},p(\beta))$ in $\mathbb{N}^n$, 
  we have $\alpha<\beta$ if either $p(\alpha)<p(\beta)$ or $p(\alpha)=p(\beta)$ and
  $\alpha_{1}<\beta_{1}$. 
\end{dfn}

The only term order on $k[X]$ such that for all $i=1,\ldots,n$, 
its restriction to $k[X_{i},\ldots,X_{n}]$ is a product order, is the lexicographic order. 
The term order $\prec$ we used above is a product order on $k[X_{0},X]$.
In the forthcoming sections, 
will explicitly indicate each instance in which we need the term order $<$ to be a product order.

\section{An interpolation technique}\label{interpolation}

Let $A$ be a variety with $m$ irreducible components, where the $\ell$-th component, $A^{(\ell)}$, 
is a $d$-plane with minimal free variables $\{X_{j};j\in J^{(\ell)}\}$.
We assume that for all $\ell$, we have $1\notin J^{(\ell)}$.
This means that all $A^{(\ell)}$ are parallel to the hyperplane $\{X_{1}=0\}$. Let 
\begin{equation}\label{y}
  Y=p(A)\subset\mathbb{A}^1\,.
\end{equation}
Clearly, $Y$ is a finite subset in $\mathbb{A}^1$ (i.e., a Zariski-closed subset of $\mathbb{A}^1$),
and for all $\lambda\in Y$, 
the intersection $A_{\lambda}$ consists of $d$-dimensional affine planes in $\mathbb{A}^{n-1}$, 
whereas for all $\lambda\in\mathbb{A}^1-Y$, $A_{\lambda}=\emptyset$. For all $\lambda\in Y$, 
let $D(A_{\lambda})\in\mathbb{D}_{n-1}$  be the standard set of $A_{\lambda}$ 
w.r.t. the restriction of $<$ to $k[\overline{X}]$. 
In this section, we assume that the sets $D(A_{\lambda})$, $\lambda\in Y$, are given, 
and we show how they are ``stacked on each other'' to give $D(A)$. 
(This is done in the case where $<$ is a product order.)
In particular, we determine not only $E(A)$, but also all lower-dimensional artifacts in $D(A)$. 
The key operation is the following.

\begin{dfn}\label{addmap}
  Let $\overline{\mathbb{D}}_{n}$ be the set of all elements of $\mathbb{D}_{n}$ 
  containing no $1$-plane parallel to $\mathbb{N}e_{1}$, thus
  \begin{equation*}
    \overline{\mathbb{D}}_{n}=\{\delta\in\mathbb{D}_{n};\mathbb{N}e_{1}\,\,{\rm is}\,\,{\rm not}\,\,{\rm contained}
    \,\,{\rm in }\,\,\delta\}\,.
  \end{equation*}
  We define the {\rm addition map}
  \begin{equation}\label{sum}
    \begin{split}
    \overline{\mathbb{D}}_{n}\times\overline{\mathbb{D}}_{n}&\to\overline{\mathbb{D}}_{n}\\
    (\delta,\delta^\prime)&\mapsto\{\beta\in\mathbb{N}^n;p(\beta)\in p(\delta)\cup p(\delta^\prime),\nonumber\\
    &\,\,\,\beta_{1}<\#p^{-1}(p(\beta))\cap\delta+\#p^{-1}(p(\beta))\cap\delta^\prime\}\,.\nonumber
    \end{split}
  \end{equation}
\end{dfn}

This operation is commutative and associative (which justifies the name, addition), 
and the empty set $\emptyset\in\overline{\mathbb{D}}_{n}$ is neutral w.r.t. $+$. 
Further, if $\delta$ and $\delta^\prime$ are finite sets, 
then $\#(\delta+\delta^\prime)=\#\delta+\#\delta^\prime$. 
For the proofs of these remarks, 
and also of the fact that $\delta+\delta^\prime$ really lies in $\overline{\mathbb{D}}_{n}$, 
see \cite{jpaa}, Section 3.

Note that each $D(A_{\lambda})\subset\mathbb{N}^{n-1}$ can be embedded into 
$\overline{\mathbb{D}}_{n}$ via the map
\begin{equation*}
  \mathbb{N}^{n-1}\hookrightarrow\mathbb{N}^n:(\alpha_{2},\ldots,\alpha_{n})\mapsto(0,\alpha_{2},\ldots,\alpha_{n})\,.
\end{equation*}
In what follows we identify each $D(A_{\lambda})$ with its image in $\overline{\mathbb{D}}_{n}$.

\begin{thm}\label{stack}
  Let $A$ be a variety as introduced at the beginning of the present section, and define $Y$ by (\ref{y}). 
  If $<$ is a product order on $k[X]$, the standard set of $D(A)$ is given by
  \begin{equation*}
    D(A)=\sum_{\lambda\in Y}D(A_{\lambda})\,,
  \end{equation*}
  where the sum is defined by (\ref{sum}).
\end{thm}

\begin{proof}
  Take an arbitrary $\alpha=(\alpha_{2},\ldots,\alpha_{n})\in\mathbb{N}^{n-1}$ and define 
  \begin{equation*}
    Y^\prime=\{\lambda\in Y;\alpha\in D(A_{\lambda})\}\,,Y^{\prime\prime}=Y-Y^\prime\,.
  \end{equation*}
  For all $\lambda\in Y^{\prime\prime}$, let $\chi_{\lambda}\in k[X_{1}]$ be the unique polynomial such that
  \begin{itemize}
    \item $\chi_{\lambda}(\mu)=\delta_{\lambda,\mu}$ for all $\mu\in Y^{\prime\prime}$ and
    \item $\deg\chi_{\lambda}=\#Y^{\prime\prime}-1$. 
  \end{itemize}
  Thus, $\chi_{\lambda}$ is the unique interpolation polynomial taking the value $1$ 
  in $\lambda$ and the value $0$ in all other elements of $Y^{\prime\prime}$.
  
  By definition of $Y^{\prime\prime}$, for all $\lambda\in Y^{\prime\prime}$, we have $\alpha\in C(A_{\lambda})$, 
  hence there exists a monic polynomial $f_{\lambda}\in I(A_{\lambda})$ with leading exponent $\alpha$. 
  We write this polynomial as
  \begin{equation*}
    f_{\lambda}=\overline{X}^\alpha+\sum_{\beta\in\mathbb{N}^{n-1},\,\beta<\alpha}c_{\lambda,\beta}\overline{X}^\beta
  \end{equation*}
  and define
  \begin{equation*}
    \begin{split}
    f=&\overline{X}^\alpha+\sum_{\lambda\in Y^{\prime\prime}}\sum_{\beta\in\mathbb{N}^{n-1},\,\beta<\alpha}
    \chi_{\lambda}c_{\lambda,\beta}\overline{X}^\beta
    \,\,{\rm and}\\
    g=&f\prod_{\lambda\in Y^\prime}(X_{1}-\lambda)\,.
    \end{split}
  \end{equation*}
  (Note that in the definition of $f$, 
  we might as well have taken $\overline{X}^\alpha$ also into the sum over $\lambda\in Y^{\prime\prime}$,
  since $\sum_{\lambda\in Y^{\prime\prime}}\chi_{\lambda}=1$.)
  Then $g\in I(A)$, since, on the one hand, $g(\lambda,\overline{X})=0$ if $\lambda\in Y^\prime$, 
  and, on the other hand,
  $g(\lambda,\overline{X})$ is a $k$-multiple of $f_{\lambda}$ if $\lambda\in Y^{\prime\prime}$. 
  Further, we have
  \begin{equation}\label{leadexp}
    {\rm LE}(g)=(\#Y^\prime,\alpha)\,,
  \end{equation}
  as follows from the hypothesis that the term order $<$ is a product order. 
  
  For our given $\alpha$, the definition of $Y^\prime$, resp. $Y^{\prime\prime}$ implies that
  \begin{equation*}
    \begin{split}
    \,\,{\rm for}\,\,{\rm all }\,\,&\lambda\in Y^\prime\,,\#p^{-1}(p(\alpha))\cap D(A_{\lambda})=1\,\,{\rm and }\\
    \,\,{\rm for}\,\,{\rm all }\,\,&\lambda\in Y^{\prime\prime}\,,\#p^{-1}(p(\alpha))\cap D(A_{\lambda})=0\,.
    \end{split}
  \end{equation*}
  Therefore, the minimal $\alpha_{1}\in\mathbb{N}$ such that 
  $(\alpha_{1},\alpha)\in\mathbb{N}^n-\sum_{\lambda\in Y}D(A_{\lambda})$ is $\alpha_{1}=\#Y^\prime$. 
  By (\ref{leadexp}), there exists a $g\in I(A)$ whose leading exponent is $(\alpha_{1},\alpha)$. 
  Upon multiplying $g$ with an arbitrary element of $k[X_{1}]$, 
  we obtain, for all $\beta\in\mathbb{N}^n-\sum_{\lambda\in Y}D(A_{\lambda})$ such that $p(\beta)=\alpha$,
  an element of $I(A)$ with leading exponent $\beta$. 
  Since the $\alpha\in\mathbb{N}^{n-1}$ we started with was arbitrary, 
  this shows that $D(A)\subset\sum_{\lambda\in Y}D(A_{\lambda})$. 

  For the converse inclusion, i.e., $C(A)\subset\mathbb{N}^n-\sum_{\lambda\in Y}D(A_{\lambda})$,
  take an arbitrary element $\beta\in C(A)$. 
  By definition of $C(A)$, there exists a $g\in I(A)$ whose leading exponent is $\beta$. 
  We write $g$ in the following form,
  \begin{equation*}
    g=\phi(X_{1})\overline{X}^{p(\beta)}+h\,,
  \end{equation*}
  for some $\phi(X_{1})\in k[X_{1}]$, where $h\in k[X]$ collects all terms of $g$ 
  in which the powers of $\overline{X}$ are strictly smaller than $\overline{X}^{p(\beta)}$.
  For all $\lambda\in Y$, the polynomial $g_{\lambda}=g(\lambda,\overline{X})$ lies in $I(A_{\lambda})$. 
  This polynomial takes the shape
  \begin{equation*}
    g_{\lambda}=\phi(\lambda)\overline{X}^{p(\beta)}+h(\lambda,\overline{X})\,,
  \end{equation*}
  hence the leading exponent of $g_{\lambda}$ is either $p(\beta)$ (which is the case if $\phi(\lambda)\neq 0$) 
  or one of the exponents of $\overline{X}$ occurring in $h$ (which is the case if $\phi(\lambda)=0$).
  This follows from the hypothesis that the term order $<$ is a product order.

  Now we define $Y^\prime$ to be the set of all $\lambda\in Y$ such that $p(\beta)\in D(A_{\lambda})$.
  For all $\lambda\in Y^\prime$, we must have $\phi(\lambda)=0$, 
  since otherwise the leading exponent of $g_{\lambda}$, which polynomial lies in $I(A_{\lambda})$,
  would be an element of $D(A_{\lambda})$. Consider the polynomial
  \begin{equation*}
    g^\prime=\phi(X_{1})\overline{X}^{p(\beta)}\,.
  \end{equation*}
  For all $\lambda\in Y^\prime$, we have $g^\prime(\lambda,\overline{X})=0$, hence $(X_{1}-\lambda)\mid g^\prime$.
  Therefore we get a factorisation
  \begin{equation*}
    g^\prime=f\prod_{\lambda\in Y^\prime}(X_{1}-\lambda)
  \end{equation*}
  for some $f\in k[X]$. It follows that $\beta$, the leading exponent of $g^\prime$, takes the shape
  \begin{equation*}
    \beta={\rm LE}(g^\prime)={\rm LE}(f)+\#Y^\prime e_{1}\,.
  \end{equation*}
  Since, on the other hand, 
  \begin{equation*}
    \#Y^\prime=\sum_{\lambda\in Y}\#p^{-1}(p(\beta))\cap D(A_{\lambda})\,,
  \end{equation*}
  we have shown that $\beta_{1}\geq\sum_{\lambda\in Y}\#p^{-1}(p(\beta))\cap D(A_{\lambda})$, 
  hence $\beta\notin\sum_{\lambda\in Y}D(A_{\lambda})$, as claimed.
\end{proof}

Let us look at an example
which illustrates the way in which the various $D(A_{\lambda})$ are stacked on each other.

\begin{ex}\label{ex2}
  Take the lexicographic order on $\mathbb{Q}[X,Y,Z]$, with $X<Y<Z$. 
  The variety $A$ has three components $A^{(1)}$, $A^{(2)}$ and $A^{(3)}$, 
  given by the Gr\"obner bases of their ideals,
  \begin{equation*}
    \begin{split}
    I(A^{(1)})=&(X-1,Z-3)\,,\\
    I(A^{(2)})=&(X-2,Z-Y+1)\,,\\
    I(A^{(3)})=&(X-3,Y-4)\,.
    \end{split}
  \end{equation*}
  Figure \ref{ex2im1} shows the components of $A$, 
  lying in the hyperplanes $\{X=1\}$, $\{X=2\}$, $\{X=3\}$, resp.
  The minimal free variable of $A^{(1)}$ and $A^{(2)}$ is $Y$, and the minimal free variable of $A^{(3)}$ is $Z$.
  The ideal $I(A)$ is given by its Gr\"obner basis,
  \begin{equation*}
    \begin{split}
    I(A)=&(X^3-6X^2+11X-6,\\
    &YX^2-3YX+2Y-4X^2+12X-8,\\
    &ZX-3Z+YX-Y-7X+13,\\
    &2ZY-4ZX^2+12ZX-8Z-Y^2X^2+Y^2X+7YX^2-13YX)\,.
    \end{split}
  \end{equation*}
  From the Gr\"obner basis, we get $D(A)$, as depicted in Figure \ref{ex2im2}. 
  Note that 
  \begin{equation*}
    \mathbb{N}e_{2}+\mathbb{N}e_{2}+\mathbb{N}e_{3}
    =\mathbb{N}e_{2}\cup((1,0,0)+\mathbb{N}e_{2})\cup(\mathbb{N}e_{3})\cup\{(2,0,0)\}\,,
  \end{equation*}
  hence $\{(2,0,0)\}$ is the the extra cube on the axis $\mathbb{N}e_{1}$.
\end{ex}

  \begin{center}
    \begin{figure}
      \begin{picture}(400,210)
         \put(130,60){\line(-3,-2){80}}
         \put(65,5){{\footnotesize $X$}}
         \put(130,60){\line(1,0){160}}
         \put(282,47){{\footnotesize $Y$}}
         \put(130,60){\line(0,1){120}}
         \put(112,170){{\footnotesize $Z$}}
         \put(115,50){\line(1,0){150}}
         \put(115,50){\line(0,1){110}}
         \put(265,50){\line(0,1){110}}
         \put(115,160){\line(1,0){150}}
         \put(100,40){\line(1,0){150}}
         \put(100,40){\line(0,1){110}}
         \put(250,40){\line(0,1){110}}
         \put(100,150){\line(1,0){150}}
         \put(85,30){\line(1,0){150}}
         \put(85,30){\line(0,1){110}}
         \put(235,30){\line(0,1){110}}
         \put(85,140){\line(1,0){150}}
         \thicklines
         \put(165,30){\line(0,1){115}}
         \put(165,30){\line(0,-11){5}}
         \put(120,40){\line(1,1){130}}
         \put(120,40){\line(-1,-1){20}}
         \put(115,110){\line(1,0){130}}
         \put(115,110){\line(-1,0){20}}
      \end{picture}
      \caption{The variety $A$ in Example \ref{ex2}}
      \label{ex2im1}
    \end{figure}
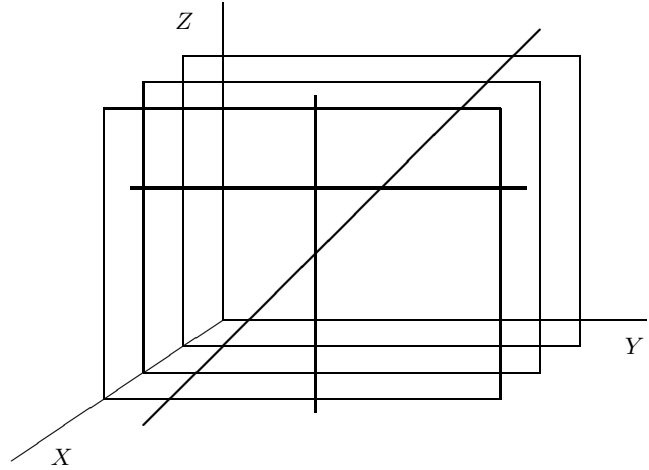
  \end{center}

\begin{center}
\begin{figure}
  \begin{picture}(400,230)
     \put(130,60){\line(-3,-2){80}}
     \put(61,4){{\footnotesize $e_{1}$}}
     \put(130,60){\line(1,0){160}}
     \put(282,47){{\footnotesize $e_{2}$}}
     \put(130,60){\line(0,1){140}}
     \put(117,190){{\footnotesize $e_{3}$}}
     \thicklines
     \put(145,75){\line(1,0){135}}
     \put(280,75){\line(0,-1){15}}
     \put(280,75){\line(-3,-2){20}}
     \put(280,60){\line(-3,-2){20}}
     \put(125,61.5){\line(1,0){135}}
     \put(125,46.5){\line(1,0){135}}
     \put(260,61.5){\line(0,-1){15}}
     \put(145,75){\line(0,1){105}}
     \put(135,68){\line(0,1){105}}
     \put(120,68){\line(0,1){105}}
     \put(135,68){\line(3,2){10}}
     \put(135,68){\line(-1,0){15}}
     \put(135,173){\line(-1,0){15}}
     \put(135,173){\line(3,2){10}}
     \put(145,180){\line(-1,0){15}}
     \put(130,180){\line(-3,-2){10}}
     \put(125,61.5){\line(-3,-2){10}}
     \put(125,61.5){\line(0,-1){15}}
     \put(125,46.5){\line(-3,-2){10}}
     \put(120,68){\line(-3,-2){20}}
     \put(114.5,54.5){\line(0,-1){15}}
     \put(114.5,54.5){\line(-1,0){15}}
     \put(99.5,54.5){\line(0,-1){15}}
     \put(99.5,39.5){\line(1,0){15}}
  \end{picture}
\caption{The standard set of $A$ in Example \ref{ex2}}
\label{ex2im2}
\end{figure}
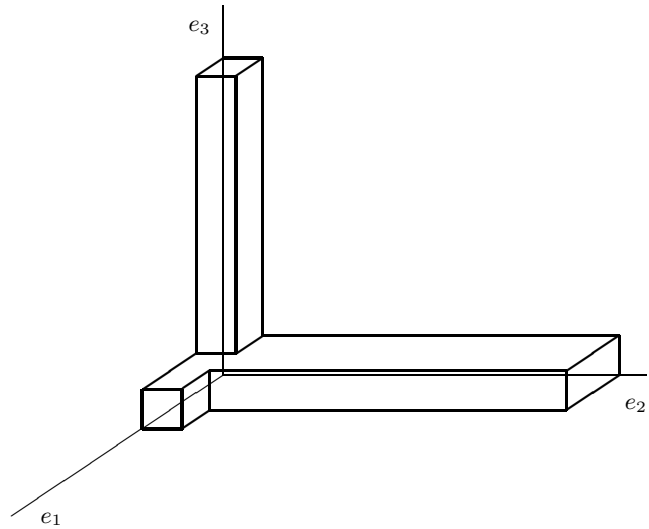
\end{center}

We now present an iterative version of Theorem \ref{stack}.
For this, we denote by $p_{0}$ our projection $p:\mathbb{A}^{n}\to\mathbb{A}^{n-1}$, 
and analogously, for all $i$, the projections
\begin{equation*}
  p_{i-1}:\mathbb{A}^{n-i+1}\to\mathbb{A}^{n-i}:(a_{i},\ldots,a_{n})\mapsto(a_{i+1},\ldots,a_{n})\,.
\end{equation*}

\begin{cor}\label{recursive}
  Let $A$ be a variety with $m$ irreducible components, where the $\ell$-th component, $A^{(\ell)}$, 
  is a $d$-plane with minimal free variables $\{X_{j};j\in J^{(\ell)}\}$.
  We assume that there exists a natural number $b\leq n-d$ such that for all $\ell$, we have $1,\ldots,b\notin J^{(\ell)}$
  and such that the restriction of the term order $<$ to $k[X_{i},\ldots,X_{n}]$ 
  is a product order for all $i=1,\ldots,b-1$. 
  We define 
  \begin{equation*}
    Y_{1}=p_{0}(A)\,,
  \end{equation*}
  and recursively for $i=1,\ldots,b$, and for all $\lambda_{i}\in Y_{i}(\lambda_{1},\ldots,\lambda_{i-1})$, 
  \begin{equation*}
    A_{\lambda_{1},\ldots,\lambda_{i}}=A\cap\{X_{1}=\lambda_{1},\ldots,X_{i}=\lambda_{i}\}\,,
  \end{equation*}
  and
  \begin{equation*}
    Y_{i+1}(\lambda_{1},\ldots,\lambda_{i})=p_{i}(A_{\lambda_{1},\ldots,\lambda_{i}})\,.
  \end{equation*}
  Then we have 
  \begin{equation*}
    A=\cup_{\lambda_{1}\in Y_{1}}\cup_{\lambda_{2}\in Y_{2}(\lambda_{1})}\ldots
    \cup_{\lambda_{b}\in Y_{b-1}(\lambda_{b-1})}A_{\lambda}\,,
  \end{equation*}
  where $\lambda=(\lambda_{1},\ldots,\lambda_{b})$ runs through all tuples defined recursively above. 
  Furthermore, 
  \begin{equation*}
    D(A)=\sum_{\lambda_{1}\in Y_{1}}\sum_{\lambda_{2}\in Y_{2}(\lambda_{1})}\sum_{\lambda_{b}\in Y_{b}(\lambda_{b-1})}
    D(A_{\lambda})\,,
  \end{equation*}
  in which sum we successively embed the elements of $\mathbb{N}^{i}$ into $\mathbb{N}^{i+1}$ via the map
  $(\alpha_{n-i+1},\ldots,\alpha_{n})\mapsto(0,\alpha_{n-i+1},\ldots,\alpha_{n})$, $i=b-1,\ldots,1$.
\end{cor}

\begin{proof}
  Apply Theorem \ref{stack} and induction over $n$. 
  For $b=1$, the assertion is identical to that of Theorem \ref{stack}.
\end{proof}

The standard set can be made completely explicit in the following class of cases.

\begin{cor}\label{corlex}
  Let $A$ be a variety as introduced above, such that
  \begin{enumerate}
    \item[(a)] either each $A_{\lambda}$ is a $d$-plane in $\mathbb{A}^{n-b}$, 
    with minimal free variables $\{X_{j};j\in J_{\lambda}\}$, where $J_{\lambda}\subset\{b+1,\ldots,n\}$,
    \item[(b)] or $d=n-b-1$ and each $A_{\lambda}$ is a union of $d$-planes in $\mathbb{A}^{n-b}$.
    In this case, we assume that for all $\lambda$ and all $i=b+1,\ldots,n$, 
    the variety $A_{\lambda}$ has $m_{\lambda,i}$ irreducible components with minimal free variables
    $\{X_{j};j\in\{b+1,\ldots,n\}-\{i\}\}$.
  \end{enumerate}
  Then $D(A_{\lambda})$ is given by 
  \begin{equation*}
    D(A)=\sum_{\lambda_{1}\in Y_{1}}\sum_{\lambda_{2}\in Y_{2}(\lambda_{1})}\sum_{\lambda_{b}\in Y_{b}(\lambda_{b-1})}
    (\oplus_{j\in J_{\lambda}}\mathbb{N}e_{j})
  \end{equation*}
  in case (a), resp. 
  \begin{equation*}
    D(A)=\sum_{\lambda_{1}\in Y_{1}}\sum_{\lambda_{2}\in Y_{2}(\lambda_{1})}\sum_{\lambda_{b}\in Y_{b}(\lambda_{b-1})}
    (\cup_{i=b+1}^n(\cup_{\ell=0}^{m_{\lambda,i}-1}(\ell e_{i}+\oplus_{j\in\{b+1,\ldots,n\}-\{i\}}\mathbb{N}e_{j})))
  \end{equation*}
  in case (b).
\end{cor}

\begin{proof}
  The assertion for case (a) follows directly from Corollary \ref{recursive}.
  The assertion for case (b) follows from Corollary \ref{recursive} and the special case $d=n-1$ 
  discussed in the context of equation (\ref{cup}).
\end{proof}

\section{Linear families of planes in affine space}\label{families}

Let $A$ be a variety with $m$ irreducible components, 
where the $\ell$-th component $A^{(\ell)}$ is a $d$-plane with minimal free variables $\{X_{j};j\in J^{(\ell)}\}$. 
We assume that for all $\ell$, we have $1\in J^{(\ell)}$.
It is easy to see that even if all components of $A$ pass through the origin of $\mathbb{A}^n$,
the irreducible components of the variety $A_{\lambda}$ will be affine, and not linear, 
$(d-1)$-planes in $\mathbb{A}^{n-1}$. This is the motivation, announced in the previous section, 
for taking assertion $\mathcal{A}(d-1,n-1)$ just as seriously as assertion $\mathcal{L}(d-1,n-1)$.

For all $\lambda\in Y$, let $\delta_{\lambda}=D(A_{\lambda})\in\mathbb{D}_{n-1}$ 
be the standard set of $A_{\lambda}$ w.r.t. the restriction of $<$ to $k[\overline{X}]$. 
For varying $\lambda$, the invariant $\delta_{\lambda}$ will in general take different values. 
Yet generically, the invariant $\delta_{\lambda}$ is constant. This will be proved in Proposition \ref{open} below.
Therefrom, we will derive Theorem \ref{inherit}, 
which states in particular that $\mathbb{N}e_{1}\oplus\delta$ is a subset of $D(A)$.
For proving Proposition \ref{open},
we have to describe the family $(A_{\lambda})_{\lambda\in\mathbb{A}^1}$ as a morphism of schemes.

Let $B=k[X_{1}]$ be the coordinate ring of $\mathbb{A}^1$. The ideal 
\begin{equation*}
  I=I(A)=\{f\in k[X];f(a)=0\,,\forall a\in A\}
\end{equation*}
defining $A$ as a subvariety of $\mathbb{A}^n$ can also be understood as an ideal
\begin{equation}\label{otheri}
  I=\{f\in B[\overline{X}];f(a)=0\,,\forall a\in A\}
\end{equation}
in the ring $B[\overline{X}]=B[X_{2},\ldots,X_{n}]$. 
Therefore, the canonical map $B\to B[\overline{X}]/I$ yields a morphism of affine schemes,
\begin{equation*}
  \phi:A={\rm Spec}\,(B[\overline{X}]/I)\to{\rm Spec}\,(B)=\mathbb{A}^1\,.
\end{equation*}
The underlying space of the fibre of $\phi$ in the point $\lambda\in\mathbb{A}^1$ 
is precisely the intersection $A_{\lambda}$ defined above. In this sense, 
the affine morphism $\phi$ is nothing but the family of affine varieties $(A_{\lambda})_{\lambda\in\mathbb{A}^1}$.

For proving Proposition \ref{open} below, we have to give a short overview over the main objects of \cite{wibmer}.
This article deals exactly with situations like the one we encounter here, 
but for more general rings $B$ and ideals $I\subset B[\overline{X}]$.
More precisely, in \cite{wibmer}, the ring $B$ is an arbitrary noetherian and reduced ring, 
and $I$ is an arbitrary ideal in the polynomial ring $B[\overline{X}]$. 
For the time being, let us describe this more general situation; 
afterwards, we will return to our particular problem. The fibres of the morphism 
\begin{equation*}
  \phi:{\rm Spec}\,(B[\overline{X}]/I)\to{\rm Spec}\,(B)\,,
\end{equation*}
and the Gr\"obner bases of these fibres, are the objects of study in \cite{wibmer}.
The aim of this article is to decompose the parameter space ${\rm Spec}\,(B)$ in such a way that
on each part $Y$ of the decomposition, the Gr\"obner bases of the fibres $\phi^{-1}(\mathfrak{p})$, 
where $\mathfrak{p}$ runs through $Y$,
come from a finite set of global sections of a certain quasi-coherent sheaf $\mathcal{I}_{Y}$ on $Y$. 
The fibres of $\phi$ can be described in the following way. For each prime ideal $\mathfrak{p}$ in $B$, 
denote by $B\to k(\mathfrak{p})$ the canonical map to the residue field. 
This map induces a homomorphism $\sigma_{\mathfrak{p}}:B[\overline{X}]\to k(\mathfrak{p})[\overline{X}]$. 
The fibre $\phi^{-1}(\mathfrak{p})$ is the subvariety of $\mathbb{A}^{n-1}_{k(\mathfrak{p})}$ defined 
by the ideal $(\sigma_{\mathfrak{p}}(I))$ in $k(\mathfrak{p})[\overline{X}]$.
Given any term order on the set of monimials in $\overline{X}$, 
we can compute leading terms, exponents, Gr\"obner bases etc. 
over $B[\overline{X}]$, and over all $k(\mathfrak{p})[\overline{X}]$. 
In particular, the standard set $D(\phi^{-1}(\mathfrak{p}))$ is a well-defined object
for each $\mathfrak{p}\in{\rm Spec}\,(B)$. We denote this set by $\delta_{\mathfrak{p}}$. 

The key technique of \cite{wibmer} is to define the quasi-coherent sheaf $\mathcal{I}_{Y}$
on each locally closed part $Y$ of ${\rm Spec}\,(B)$. This sheaf is defined as follows. 
Let $\mathfrak{a}\subset B$ be the ideal defining the closure $\overline{Y}$ of $Y$ in ${\rm Spec}\,(B)$, 
and let $\overline{I}$ be the image of $I$ in $(B/\mathfrak{a})[\overline{X}]$. 
The set $\overline{I}$ is clearly a $B/\mathfrak{a}$-module, hence defines a quasi-coherent sheaf on 
${\rm Spec}\,(B/\mathfrak{a})=\overline{Y}$. 
Now $\mathcal{I}_{Y}$ is defined to be the restriction of this quasi-coherent sheaf to $Y$.
A section $g$ of $\mathcal{I}_{Y}$ over $Y$ is a function which is locally, on an open $U$ in $Y$, 
a fraction $g=P/s$, where $P\in\overline{I}$ and $s\in(B/\mathfrak{a})-\mathfrak{q}$, for all $\mathfrak{q}\in U$. 
In particular, each section $g\in\mathcal{I}_{Y}(Y)$ can be reduced modulo $\mathfrak{p}$, 
for all $\mathfrak{p}\in Y$. We denote the reduction by $\overline{g}^{\mathfrak{p}}$. 
This is an element of the ideal $(\sigma_{\mathfrak{p}}(I))\subset k(\mathfrak{p})[\overline{X}]$.

Now that we have given an overview of the basic objects of \cite{wibmer}, 
we can apply the main result of this paper to our situation.
We return to $B=k[X_{1}]$ and $I$ as in (\ref{otheri}). In this situation,
the following proposition follows easily from Theorem 11 of \cite{wibmer}.

\begin{pro}\label{open}
  There exists an open part $U$ of $\mathbb{A}^1$ such that $\delta_{\lambda}\in\mathbb{D}_{n-1}$ 
  is constant for all $\lambda\in U$. 
\end{pro}

In this sense, the standard set $\delta$ of the proposition is the {\it generic} standard set of the family 
$(A_{\lambda})_{\mathbb{A}^1}$. This generic $\delta$ will be used throughout the rest of the paper. 
In Theorem (\ref{inherit}) below, we will show that the ``cuboid'' 
\begin{equation*}
  \mathbb{N}e_{1}\oplus\delta=\{\alpha\in\mathbb{N}^n;\alpha_{1}\in\mathbb{N},p(\alpha)\in\delta\} 
\end{equation*}
over $\delta$ is contained in $D(A)$.
For doing so, we will need a lemma, which in turn requires the following class of polynomials:

\begin{lmm}\label{fbeta}
  Let $I$ be an ideal in $k[X]$ and $\beta\in C(I)$. Then there exists a unique $f_{\beta}\in I$ such that
  \begin{itemize}
    \item $f_{\beta}$ is monic,
    \item ${\rm LE}\,f_{\beta}=\beta$, and
    \item all nonleading exponents of $f_{\beta}$ lie in $D(A)$.
  \end{itemize}
  Furthermore, the collection of all $f_{\beta}$, where $\beta$ runs through $C(A)$, is a $k$-basis of $I$.
\end{lmm}

Since this lemma is apparently well--known to experts, we skip its proof; 
however, we could not find a reference for it in the literature.
(A way of proving the kemma is to use induction over the elements of $C(A)$, 
in a similar fashion as is used in the proof of Lemma \ref{rational} below. 
If $\beta$ is the minimal element of $C(A)$, or more generally, 
a {\it corner} of $C(A)$, as defined in the proof of Lemma \ref{rational} below, 
the polynomial $f_{\beta}$ is the unique element of the reduced Gr\"obner basis with leading exponent $\beta$.)
Note that the polynomials $f_{\beta}$ are interesting in their own right, 
forming ``the'' canonical basis of the $k$-vector space $I$.

\begin{lmm}\label{rational}
  Let $I^{(\ell)}$, $\ell=1,\ldots,m$, be ideals in $k[X]$ and $I=\cap_{\ell=1}^mI^{(\ell)}$. 
  Assume that $I^{(\ell)}$ is generated by polynomials 
  \begin{equation*}
    f^{(\ell)}_{b}=\sum_{\beta\in\mathbb{N}^n}c^{(\ell)}_{b,\beta}X^{\beta}\,,
  \end{equation*}
  where $b$ runs through some indexing set $B^{(\ell)}$.
  Then the coefficients of all $f_{\beta}\in I$ are $\mathbb{Z}$-rational functions in the parameters 
  $c^{(\ell)}_{b,\beta}$, for $\ell=1,\ldots,m$, $b\in B^{(\ell)}$, $\beta\in\mathbb{N}^n$.
  
  In particular, let $\overline{A}$ be a variety in $\mathbb{A}^{n-1}={\rm Spec}\,k[\overline{X}]$ with $m$ components,
  where the $\ell$-th component is an affine $(d-1)$-plane defined by equations
  \begin{equation}\label{linearbar}
    X_{i}+\sum_{j\in J^{(\ell)},j<i}\overline{b}^{(\ell)}_{i,j}X_{j}+\overline{c}^{(\ell)}_{j}=0\,,
  \end{equation}
  for $i\in\{2,\ldots,n\}-J^{(\ell)}$, as in Proposition \ref{unique}.
  Then the coefficients of all $f_{\beta}\in I(\overline{A})$ are $\mathbb{Z}$-rational functions in the parameters 
  $\overline{b}^{(\ell)}_{i,j}$ and $\overline{c}^{(\ell)}_{j}$.
\end{lmm}

\begin{proof}
  The second part of the lemma is an immediate consequence of the first, 
  since $I(\overline{A})=\cap_{\ell=1}^mI(\overline{A}^{(\ell)})$,
  where $I(\overline{A}^{(\ell)})$ is the ideal generated by the polynomials in (\ref{linearbar}).
  
  For proving the first part of the lemma, we make the following reduction:
  It suffices to prove the statement for all elements of the reduced Gr\"obner basis of $I$. 
  Note that the reduced Gr\"obner basis of $I$ is the set of all $f_{\beta}$, 
  where $\beta$ runs through the {\it corners} of $C(I)$,
  i.e. all $\beta\in C(I)$ such that for all $i$, we have $\beta-e_{i}\notin C(I)$.
  Assume that the statement is shown for all $f_{\beta}$ in the reduced Gr\"obner basis.
  We prove that the statement is true for all $f_{\beta}$, where $\beta\in C(A)$, by induction over $\beta\in C(I)$. 
  If $\beta$ is the minimal element of $C(I)$, or more generally, a corner of $C(I)$,
  then $f_{\beta}$ is an element of the reduced Gr\"obner basis,
  and the statement is true for $f_{\beta}$ by assumption.
  If $\beta$ is a nonminimal element of $C(A)$, and more specifically, not a corner of $C(I)$, 
  then there exists an $i$ such that $\beta^\prime=\beta-e_{i}$ also lies in $C(I)$. 
  Thus in particular $\beta^\prime<\beta$, and by our induction hypothesis, 
  we may assume that the statement is true for the polynomial $f_{\beta^\prime}$. 
  We write this polynomial in the following form,
  \begin{equation*}
    f_{\beta^\prime}=X^{\beta^\prime}+\sum_{\gamma^\prime\in D(I)\,,\gamma^\prime<\beta^\prime}
    c_{\gamma^\prime}X^{\gamma^\prime}\,.
  \end{equation*}
  The product $X_{i}f_{\beta^\prime}$ lies in $I$, is monic and has leading exponent $\beta$. 
  Furthermore, the statement of the lemma clearly holds for this product.
  But it may happen that some terms of $X_{i}f_{\beta^\prime}$ with exponents in $C(I)$ do not vanish.
  If so, these exponents lie in the set
  \begin{equation*}
    \Gamma=\{\gamma^\prime+e_{i};\,\gamma^\prime\in D(I),\,\gamma^\prime<\beta^\prime\}\,.
  \end{equation*}
  For getting rid of the corresponding terms in $X_{i}f_{\beta^\prime}$, 
  we first note that for all $\gamma=\gamma^\prime+e_{i}\in\Gamma$, 
  we have $\gamma<\beta^\prime+e_{i}=\beta$. 
  Therefore, by our induction hypothesis, the statement of the lemma is true for all $f_{\gamma}$, 
  where $\gamma\in\Gamma$. Hence the statement of the lemma also holds for the polynomial 
  \begin{equation}\label{corrected}
    X_{i}f_{\beta^\prime}-\sum_{\gamma\in\Gamma}c_{\gamma-e_{i}}f_{\gamma}\,.
  \end{equation}
  Furthermore, this polynomial lies in $I$, is monic, has leading exponent $\beta$, 
  and all its nonleading exponents lie in $D(I)$. 
  Hence this polynomial equals $f_{\beta}$. The induction step is done.
  
  Now we have to prove the first statement of the lemma for all elements of the reduced Gr\"obner basis of $I$. 
  First assume we are given an arbitrary set of generators of $I$, call it $G$, 
  such that the coefficients of all elements of $G$ are $\mathbb{Z}$-rational functions in the parameters 
  $c^{(\ell)}_{b,\beta}$.
  Recall that the reduced Gr\"obner basis of $I$ is computed from $G$ by means of the Buchberger algorithm.
  In very brief terms, the Buchberger algorithm is based on two operations, 
  namely, forming $S$-pairs of elements of $G$, 
  and reducion modulo subsets of $G$. In both operations, $G$ is replaced by a set $G^\prime$, 
  where the coefficients of all elements of $G^\prime$ are 
  $\mathbb{Z}$-rational functions in the coefficients of elements of $G$. 
  For details on the Buchberger algorithm, see Chapter 2 of \cite{cox}.

  Therefore, it remains to show that there exists a system of generators $G$ of the ideal $I$, 
  say, a Gr\"obner basis of $I$, such that the coefficients of all elements of $G$ are 
  $\mathbb{Z}$-rational functions in the parameters $c^{(\ell)}_{b,\beta}$. 
  For this, we now discuss in which way we obtain a Gr\"obner basis of an intersection $I^\prime\cap I^{\prime\prime}$ 
  when given generators of $I^\prime$ and generators of $I^{\prime\prime}$; 
  a Gr\"obner basis of $I=\cap_{\ell=1}^mI^{(\ell)}$ is obtained by this token and induction over $m$.
  We introduce a new variable $T$ over $k$ and extend our term order given on $k[X]$ to $k[T,X]$ 
  by defining that $T$ be larger than any power of $X$. 
  By Theorem 11 in Chapter 4 of \cite{cox}, we have
  \begin{equation*}
    I^\prime\cap I^{\prime\prime}=(TI^\prime+(1-T)I^{\prime\prime})\cap k[X]\,,
  \end{equation*}
  where $(TI^\prime+(1-T)I^{\prime\prime})$ is the ideal generated by all products $Tf$, $f\in I^\prime$, 
  and $(1-T)g$, $g\in I^{\prime\prime}$.
  Furthermore, by the elimination theorem (see Chapter 3 of \cite{cox}), 
  we obtain a Gr\"obner basis of $(TI^\prime+(1-T)I^{\prime\prime})\cap k[X]$ by first computing a Gr\"obner basis
  of $(TI^\prime+(1-T)I^{\prime\prime})\subset k[T,X]$ w.r.t. our extension of the term order $<$
  (which is done by means of the Buchberger algorithm) 
  and subsequently picking those elements of the Gr\"obner basis which lie in $k[X]$.
  Thus, all we need for computing our desired $G$ out of the generators $f^{(\ell)}_{b}$ is the Buchberger algorithm.
  Therefore, we can use the same argument as above once more, and the lemma is proved. 
\end{proof}

Note that the statement of the lemma is not true for arbitrary $f\in I$ (and not even for arbitrary monic $f\in I$, say). 
Indeed, if $c$ is an element of $k$ which is not $\mathbb{Z}$-rational in the various $c^{(\ell)}_{b,\beta}$, 
then $cf_{\beta}$, for $\beta\in C(A)$ 
(or $f_{\beta}+cf_{\beta^\prime}$, for $\beta$ and $\beta^\prime\in C(A)$, where $\beta^\prime<\beta$)
does not have the property stated in the lemma. 

\begin{thm}\label{inherit}
  Let $A$ be a variety as introduced at the beginning of the present section, 
  and let $\delta$ be the generic standard set of Proposition \ref{open}.
  If $<$ is a product order, the set $\mathbb{N}e_{1}\oplus\delta$ is a subset of $D(A)$, 
  and is the largest subset of $D(A)$ which is a union of $1$-planes parallel to $\mathbb{N}e_{1}$.
\end{thm}

\begin{proof}
  First we construct, for each $\beta\in\mathbb{N}^{n-1}-\delta$, 
  a polynomial $h_{\beta}\in I(A)$ whose leading exponent lies in $\mathbb{N}e_{1}+\beta$.
  Therefrom will follow that $D(A)$ contains no larger subset containing a $1$-plane parallel to $\mathbb{N}e_{1}$ 
  than $\mathbb{N}e_{1}\oplus\delta$.
  
  The parameters of the irreducible components $A_{\lambda}^{(1)},\ldots,A_{\lambda}^{(m)}$ 
  of the variety $A_{\lambda}$ depend affine-linearly on $\lambda$. 
  More precisely, the equations defining $A_{\lambda}^{(\ell)}$ are
  \begin{equation*}
    X_{i}+\sum_{j\in\overline{J}^{(\ell)},\,j<i}b_{i,j}^{(\ell)}X_{j}+(c_{i}^{(\ell)}+b_{i,1}^{(\ell)}\lambda)=0
    \,\,{\rm for}\,\,{\rm all}\,\,\,i\in\{2,\ldots,n\}-\overline{J}^{(\ell)}\,.
  \end{equation*}
  Therefore, the parameters of $A_{\lambda}^{(\ell)}$ are $b^{(\ell)}_{i,j}$ and $c_{i}^{(\ell)}+b_{i,1}^{(\ell)}\lambda$, 
  for all $i\in\{2,\ldots,n\}-\overline{J}^{(\ell)}$ and $j\in\overline{J}^{(\ell)}$, $j<i$.
  
  Let $U\subset\mathbb{A}^1$ be as in Proposition \ref{open}.
  For all $\lambda\in U$ and all $\beta\in\mathbb{N}^{n-1}-\delta$, 
  denote by $f_{\lambda,\beta}$ the unique element of $I(A_{\lambda})$ introduced in Definition \ref{fbeta}. 
  By Lemma \ref{rational}, the coefficients of $f_{\lambda,\beta}$ are $\mathbb{Z}$-rational
  in the parameters $b^{(\ell)}_{i,j}$ and $c_{i}^{(\ell)}+b_{i,1}^{(\ell)}\lambda$. 
  We replace each $\lambda$ by $X_{1}$ and call the result $f_{X_{1},\gamma}$. 
  This is an element of $k(X_{1})[X_{2},\ldots,X_{n}]$, 
  where the denominators of the coefficients are products of the various $c_{i}^{(\ell)}+b_{i,1}^{(\ell)}X_{1}$. 
  Let us denote by $\overline{<}$ the term order on $k(X_{1})[\overline{X}]$ which is obtained by 
  restricting the term order $<$ on $k[X]$ to $k[\overline{X}]$
  and subsequently extending this term order trivially to $k(X_{1})[\overline{X}]$. 
  It is clear that the leading exponent w.r.t. $\overline{<}$ of $f_{X_{1},\beta}$
  equals ${\rm LE}(f_{\lambda,\beta})=\beta$. 
  
  Next, we clear the denominators of $f_{X_{1},\beta}$ and get an element $g_{\beta}$ of $k[X]$. 
  Clearly, $g_{\beta}$ vanishes on $A\cap(U\times\mathbb{A}^{n-1})$. 
  Further, the hypothesis that $<$ is a product order implies that the leading exponent of 
  $g_{\beta}$ w.r.t. $<$ arises from the leading exponent of $f_{X_{1},\beta}$ 
  w.r.t. $\overline{<}$ by addition of some element of $\mathbb{N}e_{1}$.
  The product
  \begin{equation*}
    h_{\beta}=g_{\beta}\prod_{\lambda\in\mathbb{A}^1-U}(X_{1}-\lambda)
  \end{equation*} 
  vanishes of all of $A$, thus $h_{\beta}\in I(A)$. 
  The leading exponent of $h_{\beta}$ arises from that of $g_{\beta}$ by adding
  $\#(\mathbb{A}^1-U)e_{1}$ to it. Therefore, ${\rm LE}(h_{\beta})\in\mathbb{N}e_{1}+\beta$,
  as required.
  
  It remains to show that $\mathbb{N}e_{1}\oplus\delta$ is contained in $D(A)$. Otherwise, 
  there exists a polynomial $g\in I(A)$ whose leading exponent, 
  call it $\beta$, lies in $\mathbb{N}e_{1}\oplus\delta$. 
  As in the proof of Theorem \ref{stack}, we write $g$ as 
  \begin{equation*}
    g=\phi(X_{1})\overline{X}^{p(\beta)}+h\,,
  \end{equation*}
  where $h\in k[X]$ collects all terms of $g$ in which the powers of $\overline{X}$ 
  are strictly smaller than $\overline{X}^{p(\beta)}$.
  For all $\lambda\in U$, the leading exponent of the polynomial $g_{\lambda}=g(\lambda,\overline{X})$ 
  is either $p(\beta)$ (which is the case if $\phi(\lambda)\neq 0$) 
  or one of the exponents of $\overline{X}$ occurring in $h$ (which is the case if $\phi(\lambda)=0$). 
  Since the polynomial $\phi$ has only finitely many zeros, there exist a Zariski open $U^\prime\subset U$ 
  such that for all $\lambda\in U^\prime$, the leading exponent of $g_{\lambda}$ is $p(\beta)$.
  Hence for all $\lambda\in U^\prime$, 
  we have found a polynomial $g_{\lambda}\in I(A_{\lambda})$ whose leading exponent lies in $\delta$.
  But $\delta=D(A_{\lambda})$ for all $\lambda\in U$, hence $g_{\lambda}=0$ for all $\lambda\in U^\prime$.
  Therefore, for all $\lambda\in U$, $(X_{1}-\lambda)$ divides $g$, hence $g=0$, a contradiction.
\end{proof}

Note that this theorem is a substantial refinement of Theorem \ref{number} for at least two reasons.
Firstly, the cuboid over $\delta$ (the generic $D(A_{\lambda})$) is contained in $D(A)$. 
In particular, $D(A)$ inherits all lower-dimensional artifacts of $D(A_{\lambda})$.
Secondly, the theorem implies that $D(A)$ contains no $d^\prime$-plane parallel to $\oplus_{j\in J}\mathbb{N}e_{j}$, 
where $d^\prime\leq d$ and $1\in J^\prime$.
Thus we now have better knowledge of what is contained in $D(A)$ and what is not.

Let us study an example. As remarked at the end of Section \ref{highest}, 
lower-dimensional artifacts do not occur if either the $A_{\lambda}$ are zero-dimensional
or consisting of hyperplanes of $\{X_{1}=\lambda\}$. Therefore, 
an example for a variety $A$ which inherits of lower-dimensional artifacts from $A_{\lambda}$ 
lives at least in ambient space $\mathbb{A}^4$. 
In particular, such an example is not all too vivid in the visual sense.
We chose to present an example which lives in $\mathbb{A}^3$
and shows the inheritance of the generic $D(A_{\lambda})$, 
but not the inheritance of lower-dimensional artifacts in the generic $D(A_{\lambda})$.

\begin{ex}\label{ex3}
  Take the lexicographic order on $\mathbb{Q}[X,Y,Z]$ with $X<Y<Z$. 
  The variety $A$ has two components $A^{(1)}$ and $A^{(2)}$, 
  given by the Gr\"obner bases of their ideals,
  \begin{equation*}
    \begin{split}
    I(A^{(1)})=&(Y-X,Z-1)\,\\
    I(A^{(2)})=&(Y-2X,Z-2)\,.
    \end{split}
  \end{equation*}
  Figure \ref{ex3im1} shows the components of $A$, lying in the hyperplanes $\{Z=1\}$ and $\{Z=2\}$, resp.
  The minimal free variable of both components is $X$.
  The open set $U$ is $\mathbb{A}^1-\{0\}$, and the generic 
  $D(A_{\lambda})$ is $\delta=\{(0,0),(1,0)\}$. 
  The ideal $I(A)$ has the Gr\"obner basis
  \begin{equation*}
    I(A)=(Y^2-3YX+2X^2,ZX-Y,ZY-3Y+2X,Z^2-3Z+2)\,, 
  \end{equation*}
  hence, $D(A)$ consists of the two $1$-planes $\mathbb{N}e_{1}$ and $(0,1,0)+\mathbb{N}e_{1}$ 
  plus the isolated point $(0,0,1)$.
\end{ex}
  
  \begin{center}
    \begin{figure}
      \begin{picture}(400,160)
         \put(130,60){\line(-3,-2){80}}
         \put(65,5){{\footnotesize $X$}}
         \put(130,60){\line(1,0){160}}
         \put(282,47){{\footnotesize $Y$}}
         \put(130,60){\line(0,1){70}}
         \put(108,125){{\footnotesize $Z$}}
         \put(130,80){\line(1,0){160}}
         \put(130,80){\line(-3,-2){80}}
         \put(290,80){\line(-3,-2){80}}
         \put(50,26.5){\line(1,0){160}}
         \put(130,100){\line(1,0){160}}
         \put(130,100){\line(-3,-2){80}}
         \put(290,100){\line(-3,-2){80}}
         \put(50,46.5){\line(1,0){160}}
         \thicklines
         \put(130,100){\line(3,-1){140}}
         \put(130,100){\line(-3,1){10}}
         \put(130,80){\line(1,-1){70}}
         \put(130,80){\line(-1,1){10}}
      \end{picture}
      \caption{The variety $A$ in Example \ref{ex3}}
      \label{ex3im1}
    \end{figure}
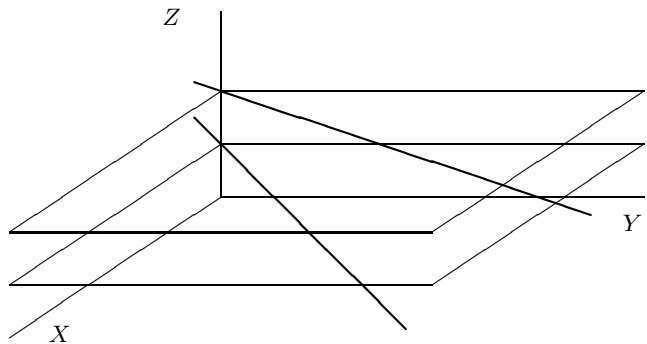
  \end{center}

\begin{center}
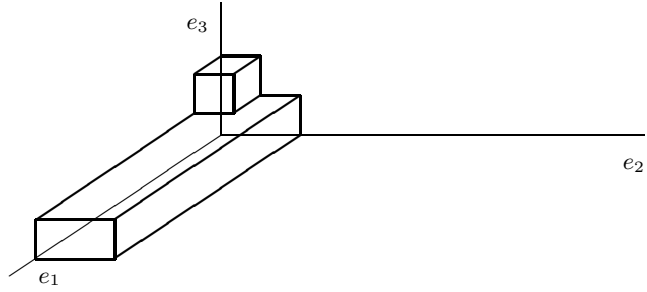
\begin{figure}
  \begin{picture}(400,140)
     \put(130,60){\line(-3,-2){80}}
     \put(61,4){{\footnotesize $e_{1}$}}
     \put(130,60){\line(1,0){160}}
     \put(282,47){{\footnotesize $e_{2}$}}
     \put(130,60){\line(0,1){50}}
     \put(117,100){{\footnotesize $e_{3}$}}
     \thicklines
     \put(145,75){\line(-3,-2){10}}
     \put(120,68){\line(-3,-2){60}}
     \put(90,13){\line(0,1){15}}
     \put(90,13){\line(-1,0){30}}
     \put(60,13){\line(0,1){15}}
     \put(60,28){\line(1,0){30}}
     \put(160,75){\line(-3,-2){70}}
     \put(160,60){\line(-3,-2){70}}
     \put(160,75){\line(0,-1){15}}
     \put(160,75){\line(-1,0){15}}
     \put(135,68){\line(0,1){15}}
     \put(135,68){\line(-1,0){15}}
     \put(145,75){\line(0,1){15}}
     \put(120,68){\line(0,1){15}}
     \put(120,83){\line(1,0){15}}
     \put(120,83){\line(3,2){10}}
     \put(145,90){\line(-1,0){15}}
     \put(145,90){\line(-3,-2){10}}
  \end{picture}
\caption{The standard set of $A$ in Example \ref{ex3}}
\label{ex3im2}
\end{figure}
\end{center}

\section{A subset of the standard set for arbitrary $A$}\label{general}

In this section, we apply the results of the previous two sections
to the study of a variety $A$ whose irreducible components are arbitrary $d$-planes, 
without any restrictions on the respective minimal free variables.

\begin{cor}\label{corgeneral}
  Let $A$ be a variety in $\mathbb{A}^n$ whose irreducible components are $d$-dimensional planes, 
  and let $<$ be a product order on $k[X]$. 
  Then there exists a $\delta\in\mathbb{D}_{n-1}$ such that the set 
  $\{\lambda\in\mathbb{A}^1;D(A_{\lambda})=\delta\}$ is dense in $\mathbb{A}^1$. 
  Let $U\subset\mathbb{A}^1$ be maximal with this property, and set $Y=\mathbb{A}^1-U$. Then 
  \begin{equation*}
    D(A)\supset(\mathbb{N}e_{1}\oplus\delta)\cup(\sum_{\lambda\in Y}D(A_{\lambda}))\,,
  \end{equation*}
  and $\mathbb{N}e_{1}\oplus\delta$ 
  is the largest subset of $D(A)$ which is a union of $1$-planes parallel to $\mathbb{N}e_{1}$.
\end{cor}

\begin{proof}
  The existence of $\delta$ and the inclusion $C(A)\subset\mathbb{N}^n-\mathbb{N}e_{1}\oplus\delta$ 
  follow by the same arguments as in the last section, in particular, the proof of Theorem \ref{inherit}.
  We define $U$ and $Y$ as in the corollary and show the inclusion
  $C(A)\subset\mathbb{N}^n-\sum_{\lambda\in Y}D(A_{\lambda})$.
  For this, we take an arbitrary $\beta\in C(A)$. We have to show that 
  \begin{equation*}
    \beta_{1}\geq\sum_{\lambda\in Y}\#p^{-1}(p(\beta))\cap D(A_{\lambda})\,.
  \end{equation*}
  In fact, we will show the following assertion, which is even stronger,
  \begin{equation}\label{strong}
    \beta_{1}\geq\sum_{\lambda\in\mathbb{A}^1}\#p^{-1}(p(\beta))\cap D(A_{\lambda})\,.
  \end{equation}
  
  There exists a $g\in I(A)$ with leading exponent $\beta$. 
  As in the proofs of Theorems \ref{stack} and \ref{inherit}, we write $g$ as 
  \begin{equation*}
    g=\phi(X_{1})\overline{X}^{p(\beta)}+h\,,
  \end{equation*}
  where $h\in k[X]$ collects all terms of $g$ in which the powers of $\overline{X}$ 
  are strictly smaller than $\overline{X}^{p(\beta)}$.
  In particular, $\beta_{1}$ equals the degree of the univariate polynomial $\phi$.
  For all $\lambda\in\mathbb{A}^1$, the leading exponent of the polynomial $g_{\lambda}=g(\lambda,\overline{X})$ 
  is either $p(\beta)$ (which is the case if $\phi(\lambda)\neq 0$) 
  or one of the exponents of $\overline{X}$ occurring in $h$ (which is the case if $\phi(\lambda)=0$). 
  We define
  \begin{equation*}
    Y^\prime=\{\lambda\in\mathbb{A}^1;p(\beta)\in D(A_{\lambda})\}\,.
  \end{equation*}
  Since for all $\lambda\in\mathbb{A}^1$, 
  the polynomial $g_{\lambda}=g(\lambda,\overline{X})$ lies in $I(A_{\lambda})$, 
  it follows that for all $\lambda\in Y^\prime$, we have $\phi(\lambda)=0$. 
  Since the number of zeros of the univariate polynomial $\phi$ is bounded by its degree, 
  which equals $\beta_{1}$, this implies that 
  \begin{equation}\label{inequality}
    \beta_{1}\geq\#Y^\prime\,. 
  \end{equation}
  By definition of $Y^\prime$, for all $\lambda$ in the complement of $Y^\prime$ in $\mathbb{A}^1$, 
  we have $p^{-1}(p(\beta))\cap D(A_{\lambda})=\emptyset$. Therefore, 
  from inequality (\ref{inequality}), the desired inequality (\ref{strong}) follows. 
  
  Finally, the fact that $\mathbb{N}e_{1}\oplus\delta$ is the largest subset of $D(A)$ which is a union of
  $1$-planes parallel to $\mathbb{N}e_{1}$ follows analogously as in the proof of Theorem \ref{inherit}.
\end{proof}

In fact, inequality (\ref{strong}) not only implies the inclusion $C(A)\subset\mathbb{N}^n-\sum_{\lambda\in Y}D(A_{\lambda})$,
which we just proved, but also the inclusion $C(A)\subset\mathbb{N}^n-\mathbb{N}e_{1}\oplus\delta$. 
Indeed, assume that $\beta$ lies in $\mathbb{N}e_{1}\oplus\delta$; 
the open set $U\subset\mathbb{A}^1$ contains infinitely many closed points $\lambda$, 
hence by (\ref{strong}) and a token using $\phi$ similarly as before, $\beta_{1}$ were to be infinitely large, a contradiction.

Here is a last example, in which the minimal free variables of the respective components of $A$ take all possible values.

\begin{ex}\label{ex4}
  We take the lexicographic order with $X<Y<Z$ on $\mathbb{Q}[X,Y,Z]$.
  The variety $A$ has five components, given by the Gr\"obner bases of the ideals,
  \begin{equation*}
    \begin{split}
    I(A^{(1)})=&(Y-X,Z-1)\,,\,\,I(A^{(4)})=(X-1,Z-3)\,,\\
    I(A^{(2)})=&(Y-X,Z-2)\,,\,\,I(A^{(5)})=(X-3,Y-4)\,.\\
    I(A^{(3)})=&(X-2,Z-Y+1)\,,
    \end{split}
  \end{equation*}
  Figure \ref{ex4im1} shows the components of $A$, along with some obvious hyperplanes in which they lie.
  The minimal free variable of $A^{(1)}$ and $A^{(2)}$ is $X$, 
  the minimal free variable of $A^{(3)}$ and $A^{(4)}$ is $Y$, 
  and the minimal free variable of $A^{(5)}$ is $Z$. 
  The open set $U$ is $\mathbb{A}^1-\{1,2,3\}$ and the generic $D(A_{\lambda})$ is $\delta=\{(0,0),(0,1)\}$. 
  The three exceptional $D(A_{\lambda})$ are 
  $D(A_{1})=\mathbb{N}e_{2}\cup\{(0,1),(0,2)\}$,
  $D(A_{2})=\mathbb{N}e_{2}\cup\{(0,1)\}$,
  $D(A_{3})=\mathbb{N}e_{3}\cup\{(1,0),(2,0)\}$
  (Note that here we are denoting the coordinate axes in $\mathbb{N}^2$ by $\mathbb{N}e_{2}$ and $\mathbb{N}e_{3}$,
  since we understand $\mathbb{N}^2$ to be the hyperplane $\{\alpha_{1}=0\}$ 
  of the ambient space $\mathbb{N}^3$ we are working in.)
  The variety $A$ is given by its Gr\"obner basis,
  \begin{equation*}
    \begin{split}
    I(A)=&(YX^3-6YX^2+11YX-6Y-X4+6X^3-11X^2+6X\,,\\
    &Y^2X^2-3Y^2X+2Y^2-YX^3-YX^2+10YX-8Y+4X^3\\
    &-12X^2+8X\,,\\
    &ZYX-3ZY-ZX^2+3ZX+Y^2X-Y^2-YX^2-6YX+13Y\\
    &+7X^2-13X\,,\\
    &ZY^2-ZYX-4ZY+4ZX-Y^3X+Y^3+Y^2X^2+7Y^2X-11Y^2\\
    &-8YX^2-5YX+28Y+16X^2-28X\,,\\
    &Z^2X^2-4Z^2X+3Z^2+ZY^2X^3-4ZY^2X^2+5ZY^2X\\
    &-2ZY^2-ZYX^4-ZYX^3+18ZYX^2-33ZYX+17ZY+5ZX^4\\
    &-23ZX^3+32ZX^2-5ZX-9Z-2Y^2X^3+10Y^2X^2-14Y^2X+6Y^2\\
    &+2YX^4-42YX^2+88YX-48Y-10X^4+56X^3-92X^2+40X+6\,,\\
    &2Z^2Y-3Z^2X+Z^2-ZY^2X^2+ZY^2X+ZYX^3+4ZYX^2\\
    &-8ZYX-3ZY-5ZX^3+8ZX^2+6ZX-3Z+2Y^2X^2-4Y^2X\\
    &+2Y^2-2YX^3-6YX^2+24YX-16Y+10X^3-26X^2+14X+2\,,\\
    &Z^3X-3Z^3+Z^2Y^2X^2-3Z^2Y^2X+2Z^2Y^2-Z^2YX^3\\
    &-2Z^2YX^2+16Z^2YX-23Z^2Y+5Z^2X^3-18Z^2X^2\\
    &+20Z^2X+15Z^2-2ZY^2X^2+14ZY^2X-12ZY^2+2ZYX^3\\
    &-4ZYX^2-64ZYX+108ZY-10ZX^3+76ZX^2-106ZX-24Z\\
    &-12Y^2X+12Y^2+12YX^2+48YX-96Y-60X^2+96X+12)\,,
    \end{split}
  \end{equation*}
  yielding the standard set $D(A)$, depicted in Figure \ref{ex4im2}.
\end{ex}

  \begin{center}
    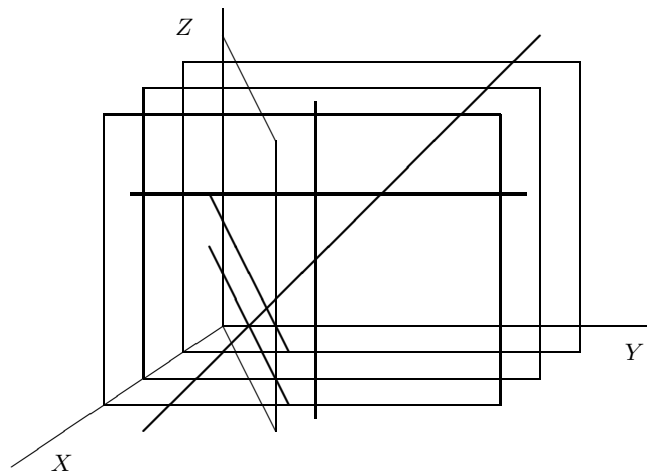
\begin{figure}
      \begin{picture}(400,210)
         \put(130,60){\line(-3,-2){80}}
         \put(65,5){{\footnotesize $X$}}
         \put(130,60){\line(1,0){160}}
         \put(282,47){{\footnotesize $Y$}}
         \put(130,60){\line(0,1){120}}
         \put(112,170){{\footnotesize $Z$}}
         \put(115,50){\line(1,0){150}}
         \put(115,50){\line(0,1){110}}
         \put(265,50){\line(0,1){110}}
         \put(115,160){\line(1,0){150}}
         \put(100,40){\line(1,0){150}}
         \put(100,40){\line(0,1){110}}
         \put(250,40){\line(0,1){110}}
         \put(100,150){\line(1,0){150}}
         \put(85,30){\line(1,0){150}}
         \put(85,30){\line(0,1){110}}
         \put(235,30){\line(0,1){110}}
         \put(85,140){\line(1,0){150}}
         \put(130,60){\line(1,-2){20}}
         \put(130,170){\line(1,-2){20}}
         \put(150,20){\line(0,1){110}}
         \thicklines
         \put(165,30){\line(0,1){115}}
         \put(165,30){\line(0,-11){5}}
         \put(120,40){\line(1,1){130}}
         \put(120,40){\line(-1,-1){20}}
         \put(115,110){\line(1,0){130}}
         \put(115,110){\line(-1,0){20}}
         \put(130,80){\line(1,-2){25}}
         \put(130,80){\line(-1,2){5}}
         \put(130,100){\line(1,-2){25}}
         \put(130,100){\line(-1,2){5}}
      \end{picture}
      \caption{The variety $A$ in Example \ref{ex4}}
      \label{ex4im1}
    \end{figure}
  \end{center}

\begin{center}
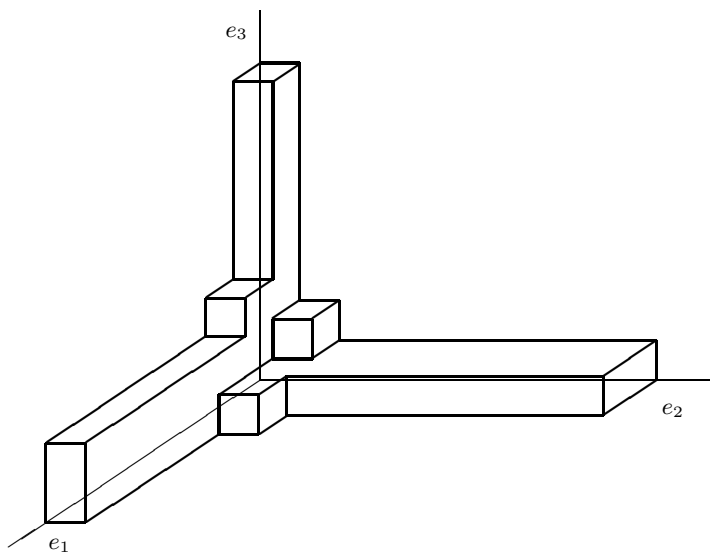
\begin{figure}
  \begin{picture}(400,240)
     \put(130,70){\line(-3,-2){95}}
     \put(50,6){{\footnotesize $e_{1}$}}
     \put(130,70){\line(1,0){170}}
     \put(282,57){{\footnotesize $e_{2}$}}
     \put(130,70){\line(0,1){140}}
     \put(117,200){{\footnotesize $e_{3}$}}
     \thicklines
     \put(160,85){\line(1,0){120}}
     \put(280,85){\line(0,-1){15}}
     \put(280,85){\line(-3,-2){20}}
     \put(280,70){\line(-3,-2){20}}
     \put(140,71.5){\line(1,0){120}}
     \put(140,56.5){\line(1,0){120}}
     \put(260,71.5){\line(0,-1){15}}
     \put(160,85){\line(-3,-2){10}}
     \put(160,85){\line(0,1){15}}
     \put(160,100){\line(-1,0){15}}
     \put(160,100){\line(-3,-2){10}}
     \put(145,100){\line(-3,-2){10}}
     \put(135,93){\line(1,0){15}}
     \put(135,93){\line(0,-1){15}}
     \put(150,78){\line(0,1){15}}
     \put(150,78){\line(-1,0){15}}
     \put(140,71.5){\line(0,-1){15}}
     \put(140,71.5){\line(-3,-2){10}}
     \put(140,56.5){\line(-3,-2){10}}
     \put(114.5,49.5){\line(1,0){15}}
     \put(129.5,49.5){\line(0,1){15}}
     \put(129.5,64.5){\line(-1,0){15}}
     \put(114.5,64.5){\line(0,-1){15}}
     \put(114.5,64.5){\line(3,2){20}}
     \put(145,100){\line(0,1){90}}
     \put(135,108){\line(0,1){75}}
     \put(120,108){\line(0,1){75}}
     \put(135,183){\line(-1,0){15}}
     \put(135,183){\line(3,2){10}}
     \put(145,190){\line(-1,0){15}}
     \put(130,190){\line(-3,-2){10}}
     \put(135,108){\line(-3,-2){10}}
     \put(135,108){\line(-1,0){15}}
     \put(120,108){\line(-3,-2){10}}
     \put(109.5,101){\line(1,0){15}}
     \put(109.5,101){\line(0,-1){14.5}}
     \put(109.5,86.5){\line(1,0){15}}
     \put(124.5,101){\line(0,-1){14.5}}
     \put(124.5,86.5){\line(-3,-2){60}}
     \put(109.5,86.5){\line(-3,-2){60}}
     \put(114.5,49.5){\line(-3,-2){50}}
     \put(49,16){\line(1,0){15}}
     \put(49,46){\line(1,0){15}}
     \put(49,16){\line(0,1){30}}
     \put(64,16){\line(0,1){30}}
  \end{picture}
\caption{The standard set of $A$ in Example \ref{ex4}}
\label{ex4im2}
\end{figure}
\end{center}

\section{Final remarks}

Interestingly, our arguments depend on $<$ to be a product order.
Philosophically, this property reflects the product decomposition $\mathbb{A}^n=\mathbb{A}^1\prod\mathbb{A}^{n-1}$. 
In all our arguments, we made use of the product decomposition of affine space, 
which explains the necessity of using a product order.

Let me make some comments on the literature concerning standard sets of finite sets in Grassmannians.
The existing approaches stress the computational aspects, 
in giving algorithms for the construction of the Gr\"obner basis of $I(A)$. 

The first article is \cite{buchmoeller}, in which an algorithm for the construction of a Gr\"obner basis of $I$ 
is constructed, where $I$ defines a finite set of closed $k$-rational points in $\mathbb{A}^n$.
(This is the case $d=0$ in our terminology.)
The idea of Buchberger--M\"oller algorithm is to successively go through the elements of $A$. 
In this iteration, one uses a control variable $\delta\in\mathbb{D}_{n}$, 
which has the property that that in each step, one has $\delta\subset D(A)$.
The stopping criterion for the Buchberger--M\"oller algorithm is that $\#A=\delta$ (hence $\delta=D(A)$).

After the appearance of the original article, 
a number of generalisations of the Buchberger--M\"oller algorithm have been published 
(see \cite{mmm}, \cite{abkr}, \cite{akr}, or the survey articles \cite{big1} and \cite{big2} and references therein).
The authors consider a $k[X]$-module $M$ and a homomorphism of $k[X]$-modules $\phi:k[X]\to M$, 
and compute the Gr\"obner basis of $\ker\phi$. 
But in fact, complete algorithms for computing the Gr\"obner basis of $I$ (and hence also of the standard set $D(I)$) 
are presented only for the cases where $\ker\phi$ defines a zero-dimensional scheme
lying either in $\mathbb{A}^n$ or in $\mathbb{P}^n$. 
Hence in our terminology, the literature covers the objects of $\mathcal{A}(0,n)$ and $\mathcal{L}(0,n)$, 
but also analogous objects with ``fat points'', i.e. points whose local ideals are powers of the associated maximal ideal.
Note that in the case where projective points are considered, 
the stopping criterion for the Buchberger--M\"oller algorithm has to be modified, 
since $\#A$ is not finite any more. The modified stopping criterion uses the Hilbert function,
in an similar way as we used it in Section \ref{highest}. 
Also the trick of intersecting with the hyperplane $\{X_{1}=1\}$ for passing from projective to affine points is being used.
See \cite{abkr} and \cite{akr} for details on the stopping criterion and the projective-to-affine trick.
A version of the interpolation technique of Section \ref{interpolation} already appear in the author's paper \cite{jpaa}. 

\section{Acknoledgements}

I wish to thank the anonymous referee of \cite{jpaa} for his positive evaluation of my article, 
which motivated me to carry on my research on this subject.
Many thanks go B. Heinrich Matzat and his group in Heidelberg, in particular Michael Wibmer, 
who gave me the opportunity to present my research in a seminar talk.
I am greatly indebted to the creators of Singular \cite{singular},
which system helped me to develop a geometric intuition for the ideas of the present research;
and to the anonymous referee of the present paper,
with the caring help of whom I was able to give it a much more concise form,
and who pointed at a large number of flaws and errors in my original manuscript. 
Many thanks indeed for sharing interest in my research, 
even though I had not presented it in an optimal way, and for the elegant arguments for proving Lemma \ref{rational}.
Special thanks go to Alexander Salle and Arne R\"uffer, 
who helped me out in a situation where I was urgently missing \cite{cox}. 

\bibliography{grass.bib}

\providecommand{\bysame}{\leavevmode\hbox to3em{\hrulefill}\thinspace}
\providecommand{\MR}{\relax\ifhmode\unskip\space\fi MR }
\providecommand{\MRhref}[2]{%
  \href{http://www.ams.org/mathscinet-getitem?mr=#1}{#2}
}
\providecommand{\href}[2]{#2}
\begin{thebibliography}{ABKR00}

\bibitem[ABKR00]{abkr}
J.~Abbott, A.~Bigatti, M.~Kreuzer, and L.~Robbiano, \emph{Computing ideals of
  points}, J. Symbolic Comput. \textbf{30} (2000), no.~4, 341--356.
  \MR{MR1784266 (2001j:13026)}

\bibitem[AKR05]{akr}
J.~Abbott, M.~Kreuzer, and L.~Robbiano, \emph{Computing zero-dimensional
  schemes}, J. Symbolic Comput. \textbf{39} (2005), no.~1, 31--49.
  \MR{MR2168239 (2006g:13050)}

\bibitem[AMM03]{big1}
Mar{\'{\i}}a~Emilia Alonso, Maria~Grazia Marinari, and Teo Mora, \emph{The big
  mother of all dualities: {M}\"oller algorithm}, Comm. Algebra \textbf{31}
  (2003), no.~2, 783--818. \MR{MR1968924 (2004b:13029)}

\bibitem[AMM06]{big2}
\bysame, \emph{The big mother of all dualities. {II}. {M}acaulay bases}, Appl.
  Algebra Engrg. Comm. Comput. \textbf{17} (2006), no.~6, 409--451.
  \MR{MR2270332 (2008d:13036)}

\bibitem[CLO05]{cox2}
David~A. Cox, John Little, and Donal O'Shea, \emph{Using algebraic geometry},
  second ed., Graduate Texts in Mathematics, vol. 185, Springer, New York,
  2005. \MR{MR2122859 (2005i:13037)}

\bibitem[CLO07]{cox}
David Cox, John Little, and Donal O'Shea, \emph{Ideals, varieties, and
  algorithms}, third ed., Undergraduate Texts in Mathematics, Springer, New
  York, 2007, An introduction to computational algebraic geometry and
  commutative algebra. \MR{MR2290010 (2007h:13036)}

\bibitem[GH78]{griff}
Phillip Griffiths and Joseph Harris, \emph{Principles of algebraic geometry},
  Wiley-Interscience [John Wiley \& Sons], New York, 1978, Pure and Applied
  Mathematics. \MR{MR507725 (80b:14001)}

\bibitem[GPS05]{singular}
G.-M. Greuel, G.~Pfister, and H.~Sch\"onemann, \emph{{\sc Singular} 3.0}, {A
  Computer Algebra System for Polynomial Computations}, Centre for Computer
  Algebra, University of Kaiserslautern, 2005, {\tt
  http://www.singular.uni-kl.de}.

\bibitem[HP94]{hodge}
W.~V.~D. Hodge and D.~Pedoe, \emph{Methods of algebraic geometry. {V}ol.
  {III}}, Cambridge Mathematical Library, Cambridge University Press,
  Cambridge, 1994, Book V: Birational geometry, Reprint of the 1954 original.
  \MR{MR1288307 (95d:14002c)}

\bibitem[Laf03]{lafforgue}
L.~Lafforgue, \emph{Chirurgie des grassmanniennes}, CRM Monograph Series,
  vol.~19, American Mathematical Society, Providence, RI, 2003. \MR{MR1976905
  (2004k:14085)}

\bibitem[Led08]{jpaa}
M.~Lederer, \emph{The vanishing ideal of a finite set of closed points in
  affine space}, J. Pure Appl. Algebra \textbf{212} (2008), 1116--1133.

\bibitem[MB82]{buchmoeller}
H.~M. M{\"o}ller and B.~Buchberger, \emph{The construction of multivariate
  polynomials with preassigned zeros}, Computer algebra (Marseille, 1982),
  Lecture Notes in Comput. Sci., vol. 144, Springer, Berlin, 1982, pp.~24--31.
  \MR{MR680050 (84b:12003)}

\bibitem[MMM93]{mmm}
M.~G. Marinari, H.~M. M{\"o}ller, and T.~Mora, \emph{Gr\"obner bases of ideals
  defined by functionals with an application to ideals of projective points},
  Appl. Algebra Engrg. Comm. Comput. \textbf{4} (1993), no.~2, 103--145.
  \MR{MR1223853 (94g:13019)}

\bibitem[Stu96]{sturm}
Bernd Sturmfels, \emph{Gr\"obner bases and convex polytopes}, University
  Lecture Series, vol.~8, American Mathematical Society, Providence, RI, 1996.
  \MR{MR1363949 (97b:13034)}

\bibitem[Wib07]{wibmer}
Michael Wibmer, \emph{Gr\"obner bases for families of affine or projective
  schemes}, J. Symbolic Comput. \textbf{42} (2007), no.~8, 803--834.
  \MR{MR2345838}

\end{thebibliography}
\bibliographystyle{amsalpha}

\end{document}